\documentclass[final]{dmtcs-episciences}

\usepackage[english]{babel}
\usepackage[utf8]{inputenc}
\usepackage[T1]{fontenc}
\usepackage{amssymb,amsmath,amsthm}
\usepackage{relsize}
\usepackage{thmtools,thm-restate}
\usepackage{color}
\usepackage{hyperref}
\usepackage{caption}
\usepackage{tikz}
\usetikzlibrary{positioning}

\usepackage[round]{natbib}
\setcitestyle{authoryear,open={((},close={))}} 

\theoremstyle{plain}
\newtheorem{thm}{Theorem}
\newtheorem{prop}[thm]{Proposition}
\newtheorem{cor}[thm]{Corollary}
\newtheorem{lem}[thm]{Lemma}

\theoremstyle{remark}
\newtheorem{rem}[thm]{Remark}
\newtheorem{ex}[thm]{Example}

\theoremstyle{definition}
\newtheorem{defi}[thm]{Definition}

\def\N{\mathbb{N}}
\def\Z{\mathbb{Z}}

\def\cA{\mathcal{A}}
\def\cB{\mathcal{B}}
\def\cC{\mathcal{C}}
\def\cD{\mathcal{D}}
\def\cE{\mathcal{E}}
\def\cG{\mathcal{G}}
\def\cL{\mathcal{L}}
\def\cR{\mathcal{R}}
\def\cS{\mathfrak{S}}
\def\cT{\mathcal{T}}
\def\eps{\varepsilon}

\def\bsigma{{\boldsymbol{\sigma}}}

\DeclareMathOperator{\Card}{Card}
\DeclareMathOperator{\Suff}{Suff}

\DeclareMathOperator{\dom}{dom}
\DeclareMathOperator{\diag}{diag}

\DeclareDocumentCommand{\longdash}{}{%
  \mathrel{\relbar\mkern-6.5mu\relbar\mkern-6.5mu\relbar}}

\title{$S$-adic characterization of minimal dendric shifts}

\author{France Gheeraert\affiliationmark{1}\thanks{Supported by a FNRS fellowship during this work.} \and Julien Leroy\affiliationmark{2}}

\affiliation{%
LAMFA, Université de Picardie Jules Verne, Amiens, France\\
UR Mathematics, University of Liège, Liège, Belgium}

\keywords{Dendric shifts, $S$-adic, substitution, Sturmian shifts, Interval exchange transformations, Arnoux-Rauzy shifts.}

\begin{document}

\publicationdata{vol. 27:2}{2025}{2}{10.46298/dmtcs.13130}{2024-02-27; 2024-02-27; 2024-11-07}{2025-01-10}

\maketitle

\begin{abstract}
Dendric shifts are defined by combinatorial restrictions of the extensions of the words in their languages. 
This family generalizes well-known families of shifts such as Sturmian shifts, Arnoux-Rauzy shifts and codings of interval exchange transformations.
It is known that any minimal dendric shift has a primitive $\mathcal{S}$-adic representation where the morphisms in $\mathcal{S}$ are positive tame automorphisms of the free group generated by the alphabet.
In this paper, we present an $\mathcal{S}$-adic characterization of this family using two finite graphs.
As an application, we are able to decide whether a shift space generated by a uniformly recurrent morphic word is (eventually) dendric.
\end{abstract}

\section{Introduction}
Dendric shifts are defined in terms of extension graphs that describe the left and right extensions of their factors. 
Extension graphs are bipartite graphs that can roughly be described as follows: if $u$ is a word in the language $\cL(X)$ of the shift space $X$, one puts an edge between the left and right copies of letters $a$ and $b$ such that $aub$ is in $\cL(X)$. 
A shift space is then said to be dendric if the extension graph of every word of its language is a tree.
These shift spaces were initially defined through their languages under the name of tree sets~\citep{acyclic} and were studied in a series of papers.
They generalize classical families of shift spaces such as Sturmian shifts~\citep{morse_hedlund}, Arnoux-Rauzy shifts~\citep{arnoux_rauzy}, codings of regular interval exchange transformations (IET)~\citep{oseledec,arnold} , and shift spaces arising from the application of the Cassaigne multidimensional continued fraction algorithm (MCF)~\citep{Cassaigne_Labbe_Leroy}.

Minimal dendric shifts exhibit striking combinatorial~\citep{bifix_decoding,rigidity}, algebraic~\citep{acyclic,finite_index} , and ergodic properties~\citep{dimension_group}. 
For instance, they have factor complexity $\#(\cL(X) \cap \cA^n) = (\#\cA-1)n+1$~\citep{acyclic} and topological rank $\#\cA$~\citep{dimension_group}, where $\cA$ is the alphabet of the shift space. 
They also fall into the class of eventually dendric shift spaces, which are exactly those satisfying the regular bispecial condition~\citep{Dolce_Perrin:2021}. This implies that the number of their ergodic measures is at most $\#\cA/2$~\citep{damron_fickenscher:2022}.
An important property for our work is that the derived shift of a minimal dendric shift is again a minimal dendric shift on the same alphabet, where derivation is here understood as derivation by return words.
This allows to give $\cS$-adic representations of such shift spaces~\citep{Ferenczi:1996}, i.e., to define a set $\cS$ of endomorphisms of the free monoid $\cA^*$ and a sequence $\bsigma = (\sigma_n)_{n \geq 0} \in \cS^\N$, called an $\cS$-adic representation, such that 
\[
    X = \{x \in \cA^{\Z} \mid u \in \cL(x) \Rightarrow \exists n\in \N, a\in \cA: u \in \cL(\sigma_0 \sigma_2 \cdots \sigma_n(a))\}.
\]

$\cS$-adic representations are a classical tool that allows to study several properties of shift spaces such as the factor complexity~\citep{Durand_Leroy_Richomme,donoso_durand_maass_petite}, the number of ergodic measures~\citep{berthe_delecroix,bedaride_hilion_lustig_1,bedaride_hilion_lustig_2}, the dimension group and topological rank~\citep{dimension_group} or yet the automorphism group~\citep{espinoza_maass}. 
In the case of minimal dendric shifts, the involved endomorphisms are particular tame automorphisms of the free group generated by the alphabet~\citep{bifix_decoding,rigidity}. 
This in particular allows to prove that minimal dendric shifts have topological rank equal to the cardinality of the alphabet and that ergodic measures are completely determined by the measures of the letter cylinders~\citep{dimension_group,bedaride_hilion_lustig_2}.

In the case of the ternary alphabet $\{1,2,3\}$, we were able to give an $\cS$-adic characterization of minimal dendric shifts~\citep{Gheeraert_Lejeune_Leroy:2021}.
We explicitly define an infinite set $\cS_3$ of positive automorphisms of the free group $F_3$ and a labeled directed graph $\cG$ with two vertices and labels in $\Sigma_3 \cS_3 \Sigma_3$ (where $\Sigma_3$ is the symmetric group of $\{1,2,3\}$) such that a shift space over $\{1,2,3\}$ is a minimal dendric shift if and only if it has a primitive $\Sigma_3 \cS_3 \Sigma_3$-adic representation labelling an infinite path in $\cG$.
We were then able to localize in $\cG$ classical families of dendric shifts such as codings of regular interval exchanges or Arnoux-Rauzy shifts.

In this paper we extend this work to any alphabet as follows.

\begin{restatable}{thm}{mainThm}\label{T:main}
Let $\cS$ be a family of dendric return morphisms from $\cA^*$ to $\cA^*$ and let $X$ be a shift space having an $\cS$-adic representation $\bsigma = (\sigma_n)_{n \geq 0}$.  
Then $X$ is minimal dendric if and only if $\bsigma$ is primitive and labels infinite paths in the graphs $\cG^L(\cS)$ and $\cG^R(\cS)$.
\end{restatable}

Observe that, contrary to the ternary case, we are not able to explicitely define the set $\cS$.
We also consider two finite graphs instead of one. 
A key argument in the ternary case is that for all $N$, the $\cS$-adic shift $X^{(N)}$ generated by $(\sigma_n)_{n \geq N}$ is minimal and dendric, so that $\sigma_N$ preserves dendricity from $X^{(N+1)}$.
The finiteness of $\cG$ was obtained by defining an equivalence relation with finite index on the set of dendric shifts such that two shifts $X,Y$ are equivalent if and only if the same morphisms of $\Sigma_3 \cS_3 \Sigma_3$ preserve dendricity from $X$ and from $Y$.
The vertices of $\cG$ are the two equivalent classes of this relation.

In the general case, while we are not able to explicitly define $\cS$, we have enough information on the morphisms in $\cS$ to similarly define an equivalence relation characterizing how dendricity is preserved.
It is defined by means of two finite graphs describing the left and right extensions of infinite special factors of $X$ (see Section~\ref{S:definition of G^L and G^R}).

The paper is organized as follows.
We start by giving, in Section~\ref{S:definitions}, the basic definitions for the study of shift spaces. We introduce the notion of extension graph of a word and of a dendric shift.
In Section~\ref{S:image under return morphism}, we recall and extend some results of~\citep{Gheeraert_Lejeune_Leroy:2021} about the stability of dendricity when taking the image under a return morphism.

In Section~\ref{S:definition of G^L and G^R}, we introduce new graphs associated with a shift space. They provide another characterization of dendric shift spaces. We also study the link with eventual dendricity in Section~\ref{S:stabilization for eventually dendric}.

After that, in Section~\ref{S:image of graphs}, we reformulate the results of Section~\ref{S:image under return morphism} and study the image under a return morphism using these graphs. We then use these results to obtain an $S$-adic characterization of dendric shift spaces (Theorem~\ref{T:main}) in Section~\ref{S:S-adic characterization}. We illustrate our results by giving an explicit graph characterizing the minimal dendric shift spaces over four letters having exactly one right special factor of each length. We also show that (eventual) dendricity, and the corresponding threshold, is decidable for substitutive shift spaces.

Finally, in Section~\ref{S:interval exchanges}, we focus on the sub-family of interval exchanges and provide an $\cS$-adic characterization (Theorem~\ref{T:iet in G}) using a subgraph of the graph obtained in the dendric case.

\section{Definitions}\label{S:definitions}

\subsection{Words, languages and shift spaces}

Let $\cA$ be a finite alphabet of cardinality $d \geq 2$ (note that some of the results do not work for unary alphabets).
Let us denote by $\varepsilon$ the empty word of the free monoid $\cA^*$ (endowed with concatenation), by $\cA^+$ the set of non-empty finite words on $\cA$ and by $\cA^{\Z}$ the set of bi-infinite words over $\cA$.
For a word  $w= w_{1} \cdots w_{\ell} \in \cA^\ell$, its  {\em length} is denoted $|w|$ and equals $\ell$.
We say that a word $u$ is a {\em factor} of a word $w$ if there exist words $p,s$ such that $w = pus$.
If $p = \varepsilon$ (resp., $s = \varepsilon$) we say that $u$ is a {\em prefix} (resp., {\em suffix}) of $w$.
For a word  $u \in  \cA^{*}$ and a word $w \in \cA^* \cup \cA^\Z$, an index $j$ such that $w_{j}\cdots w_{j+|u|-1} = u$ is called an {\em occurrence} of $u$ in $w$.

The set  $\cA^{\Z}$ endowed with the product topology of the discrete topology on each copy of $\cA$ is topologically a Cantor set. 
The {\em shift map}  $S$ defined by $S \left( (x_n)_{n \in \mathbb{Z}} \right) = (x_{n+1})_{n \in \mathbb{Z}}$ is a homeomorphism of  $\cA^{\Z}$. 
A {\em shift space} is a pair $(X,S)$ where $X$ is a closed shift-invariant subset of some $\cA^{\Z}$.
It is thus a {\em topological dynamical system}.
It is {\em minimal} if the only closed shift-invariant subset $Y \subseteq X$ are $\emptyset$ and $X$.
Equivalently, $(X,S)$ is minimal if and only if the orbit of every $x \in X$ is dense in $X$.
Usually we say that the set $X$ is itself a shift space.

The {\em language} of a sequence $x \in \cA^{\Z}$ is its set of factors and is denoted $\cL(x)$. 
For a shift space $X$, its {\em language} $\cL(X)$ is $\bigcup_{x\in X} \cL(x)$ and we set $\cL_n(X) = \cL(X) \cap \cA^n$, $n \in \N$, and $\cL_{\leq N}(X) = \bigcup_{n \leq N} \cL_n(X)$.
Its {\em factor complexity} is the function $p_X:\N \to \N$ defined by $p_X(n) = \Card \cL_n(X)$.
We say that a shift space $X$ is {\em over} $\cA$ if $\cL_1(X) = \cA$.

\subsection{Extension graphs and dendric shifts}

Dendric shifts are  defined with respect to combinatorial properties of their language expressed in terms of extension graphs. 
Let $F$ be a set of finite words on the alphabet $\cA$ which is factorial, i.e., if $u \in F$ and $v$ is a factor of $u$, then $v \in F$.
For $w \in F$, we define the sets of left, right and bi-extensions of $w$ by
\begin{align*}
	E_F^L(w) &= \{ a \in \cA \mid aw \in F\};	\\
	E_F^R(w) &= \{ b \in \cA \mid wb \in F\};	\\
	E_F(w) &= \{ (a,b) \in \cA \times \cA \mid awb \in F\}.		
\end{align*}
The elements of $E_F^L(w)$, $E_F^R(w)$ and $E_F^L(w)$ are respectively called the {\em left extensions}, the {\em right extensions} and the {\em bi-extensions} of $w$ in $F$. If $F$ is the language of a shift space $X$, we instead use the terminology \emph{extensions in $X$} and the index $F$ is replaced by $X$ or even omitted if the context is clear.
Observe that as $X \subseteq \cA^\Z$, the set $E_X(w)$ completely determines $E_X^L(w)$ and $E_X^R(w)$.  
A word $w$ is said to be {\em right special} (resp., {\em left special}) if $\Card(E^R(w))\geq 2$ (resp., $\Card(E^L(w)) \geq 2$). 
It is {\em bispecial} if it is both left and right special.
The factor complexity of a shift space is completely governed by the extensions of its special factors.
In particular, we have the following result.

\begin{prop}[Cassaigne and Nicolas~\citep{CANT_cassaigne}]
\label{P:complexity}
Let $X$ be a shift space. 
For all $n$, we have 
\begin{align*}
	p_X(n+1)-p_X(n) 
	&= \sum_{w \in \cL_n(X)} (\Card(E^R(w))-1)	\\
	&= \sum_{w \in \cL_n(X)} (\Card(E^L(w))-1).
\end{align*}
In addition, if for every bispecial factor $w \in \cL(X)$, one has 
\begin{equation}
\label{eq:bilateral multiplicity}
	\Card(E(w)) - \Card(E^L(w)) - \Card(E^R(w)) + 1 = 0,
\end{equation}
then $p_X(n) = (p_X(1)-1)n +1$ for every $n$. 
\end{prop}

A classical family of factors satisfying Equation~\eqref{eq:bilateral multiplicity} is made of the {\em ordinary} factors that are defined by the existence of $(a,b) \in E(w)$ such that $E(w) \subseteq (\{a\} \times \cA) \cup (\cA \times \{b\})$.
A larger family of factors also satisfying Equation~\eqref{eq:bilateral multiplicity} are the dendric factors defined below.

For a word $w \in F$, we consider the undirected bipartite graph $\cE_F(w)$ called its \emph{extension graph} with respect to $F$ and defined as follows:
its set of vertices is the disjoint union of $E_F^L(w)$ and $E_F^R(w)$ and its edges are the pairs $(a,b) \in E_F^L(w) \times E_F^R(w)$ such that $awb \in F$.
For an illustration, see Example~\ref{ex:fibo} below.
Note that a letter $a$ can label a vertex both on the left and on the right side. To distinguish the two, we will denote the left vertex $a^L$ and the right vertex $a^R$.
We say that $w$ is {\em dendric} if $\cE(w)$ is a tree. 
We then say that a shift space $X$ is a \emph{dendric shift} if all its factors are dendric in $\cL(X)$ and it is an \emph{eventually dendric shift (with threshold $N$)} if all factors of length at least $N$ are dendric (where $N$ is chosen minimal).

Note that every non-bispecial word and every ordinary bispecial word is trivially dendric.
In particular, the Arnoux-Rauzy shift spaces are dendric (recall that Arnoux-Rauzy shift spaces are the minimal shift spaces having exactly one left special factor $u_n$ and one right special factor $v_n$ of each length $n$ and such that $E^L(u_n) = \cA = E^R(v_n)$; all factors of an Arnoux-Rauzy shift are ordinary).
By Proposition~\ref{P:complexity}, we deduce that any dendric shift has factor complexity $p_X(n) = (p_X(1)-1)n +1$ for every $n$.

\begin{ex}\label{ex:fibo}
Take a shift $X$ such that $\cL_3(X) = \{001,010,100,101\}$ (for instance, the Fibonacci shift on $\{0,1\}$).
The extension graphs of the empty word and of the two letters $0$ and $1$ are represented in Figure~\ref{fig:fibo-ext}.

\begin{figure}[h]
 \tikzset{node/.style={circle,draw,minimum size=0.5cm,inner sep=0pt}}
 \tikzset{title/.style={minimum size=0.5cm,inner sep=0pt}}

 \begin{center}
  \begin{tikzpicture}
   \node[title](ee) {$\cE(\varepsilon)$};
   \node[node](eal) [below left= 0.5cm and 0.5cm of ee] {$0$};
   \node[node](ebl) [below= 0.7cm of eal] {$1$};
   \node[node](ear) [right= 1.5cm of eal] {$0$};
   \node[node](ebr) [below= 0.7cm of ear] {$1$};
   \path[draw,thick]
    (eal) edge node {} (ear)
    (eal) edge node {} (ebr)
    (ebl) edge node {} (ear);
   \node[title](ea) [right = 3cm of ee] {$\cE(0)$};
   \node[node](aal) [below left= 0.5cm and 0.5cm of ea] {$0$};
   \node[node](abl) [below= 0.7cm of aal] {$1$};
   \node[node](aar) [right= 1.5cm of aal] {$0$};
   \node[node](abr) [below= 0.7cm of aar] {$1$};
   \path[draw,thick]
    (aal) edge node {} (abr)
    (abl) edge node {} (aar)
    (abl) edge node {} (abr);
   \node[title](eb) [right = 3cm of ea] {$\cE(1)$};
   \node[node](bal) [below left= 0.5cm and 0.5cm of eb] {$0$};
   \node[node](bar) [right= 1.5cm of bal] {$0$};
   \path[draw,thick]
    (bal) edge node {} (bar);
  \end{tikzpicture}
 \end{center}

 \caption{The extension graphs of $\varepsilon$ (on the left), $0$ (in the center) and $1$ (on the right) are trees.}
 \label{fig:fibo-ext}
\end{figure}
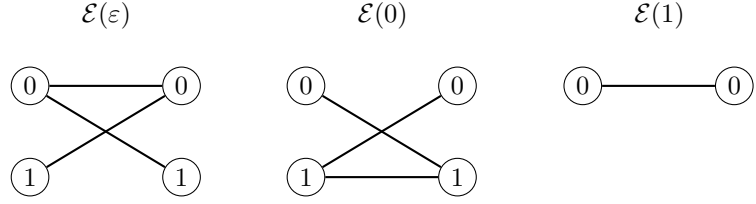
\end{ex}

\section{Extension graphs in morphic images}
\label{S:image under return morphism}

Let $\cA, \cB$ be finite alphabets with cardinality at least 2. 
By a {\em morphism} $\sigma:\cA^* \to \cB^*$, we mean a monoid homomorphism (also called a \emph{substitution} when $\cA = \cB$). We will always assume that $\cB$ is minimal, i.e. each letter appears in the image of some word.
A morphism is said to be {\em non-erasing} if the image of any letter is a non-empty word. 
We stress the fact that all morphisms are assumed to be non-erasing in the following.
Using concatenation, we extend $\sigma$ to~$\cA^\mathbb{Z}$. 
In particular, if $X$ is a shift space over $\cA$, the {\em image of $X$ under $\sigma$} is the shift space 
\[
\sigma \cdot X := \{S^k \sigma(x) \mid x \in X, 0\leq k < |\sigma(x_0)|\}.
\]

In general, it can be difficult to deduce extension graphs of words in $\cL(\sigma \cdot X)$ from those of $\cL(X)$. 
In this paper, we will restrict ourselves to particular morphisms which offer some nice recognizability properties in the sense that the pre-images of words are well understood.
These morphisms are defined using the notion of return words. Let $X$ be a shift space. If $w \in \cL(X)$ is non-empty, a \emph{return word} to $w$ in $X$ is a non-empty word $r$ such that $rw \in \cL(X)$ and $rw$ contains exactly two occurrences of $w$, one as a prefix and one as a suffix.

\begin{defi}
A \emph{return morphism for a word} $w \in \cB^+$ is an injective morphism $\sigma : \cA^* \to \cB^*$ such that, for all $a \in \cA$, $\sigma(a)w$ contains exactly two occurrences of $w$, one as a (proper) prefix and one as a (proper) suffix.
\end{defi}

In particular, the set $\sigma(\cA)$ is then a suffix code, i.e. $\sigma(a)$ is not a suffix of $\sigma(b)$ if $b \ne a$. Hence, injectivity on the letters is equivalent to injectivity on the words.

\begin{ex}
The morphism $\sigma : 0 \mapsto 01, 1 \mapsto 010, 2 \mapsto 0102$ is a return morphism for $01$ and it is also a return morphism for $010$.
\end{ex}

As we have seen in the previous example, the word $w$ for which $\sigma$ is a return morphism is not always unique. However, we have the following result.

\begin{prop}\label{P:return morphism for two words}
Let $\sigma$ be a return morphism for two distinct words $w$ and $w'$ and let $X$ be a shift space. If $|w| \leq |w'|$, then
\begin{enumerate}
\item
	$w$ is a proper prefix of $w'$;
\item
	$\sigma$ is a return morphism for all prefixes of $w'$ of length at least $|w|$;
\item
	$w$ is not right special in $\sigma \cdot X$;
\item
	if $w'$ is of maximal length, it is right special in $\sigma \cdot X$.
\end{enumerate}
\end{prop}

\begin{proof}\hfill
\begin{enumerate}
\item
    Let $a \in \cA$. 
    If $|w|\leq |\sigma(a)|$, then $w$ is a prefix of $\sigma(a)$.
    As $w'$ is also prefix of $\sigma(a)w'$ and $|w|\leq|w'|$, $w$ is a proper prefix of $w'$.
    Otherwise, $|w'|\geq|w|>|\sigma(a)|$ and both $w$ and $w'$ are of the form $(\sigma(a))^np$, where $p$ is a prefix of $\sigma(a)$. 
    Thus $w$ is a proper prefix of $w'$.
\item
	Let $wu$ be a prefix of $w'$. By definition, it is a prefix and a suffix of $\sigma(a)wu$ for all $a \in \cA$.
	Moreover, since $w$ has only two occurrences in $\sigma(a)w$, $wu$ has only two occurrences in $\sigma(a)wu$ thus $\sigma$ is a return morphism for $wu$.
\item
	Since $\sigma$ is a return morphism for $w$, $w$ only occurs in $\sigma \cdot X$ as a proper prefix of the words $\sigma(a)w$, $a \in \cA$.
	However, for all $a \in \cA$, $w'$ is a prefix of $\sigma(a)w'$ thus $ww'_{|w| + 1}$ is a prefix of $\sigma(a)w$. The only right extension of $w$ is then $w'_{|w| + 1}$.
\item
	Assume that $w'$ has only one right extension $b$. Thus $w'b$ is a prefix (and a suffix) of each $\sigma(a)w'b$, $a \in \cA$. As $w'$ has only two occurrences in $\sigma(a)w'$, the word $w'b$ has only two occurrences in $\sigma(a)w'b$, which proves that $\sigma$ is a return morphism for $w'b$, a contradiction.
\end{enumerate}
\end{proof}

We will sometimes talk about a \emph{return morphism} without specifying the word $w$ for which it is a return morphism. 
By default, this $w$ will be chosen of maximal length which is possible by the previous proposition. Because of the previous proposition, most results are independent of this choice (or only depend in a minor way).

We now give a description of the extension graphs in $\sigma \cdot X$ when $\sigma$ is a return morphism. We first introduce some notations.

\begin{defi}
Let $\sigma : \cA^* \to \cB^*$ be a return morphism for the word $w \in \cB^+$. For all $u \in \cB^*$, we define the sets
\[
	\cA^L_{\sigma, u} = \{a \in \cA \mid \sigma(a) \in \cB^+u\}
\]
and
\[
	\cA^R_{\sigma, u} = \{a \in \cA \mid \sigma(a)w \in u\cB^+\}.
\]

For a shift $X$ over $\cA$ and a word $v \in \cL(X)$, we then note $E_{X, s, p}(v) = E_X(v) \cap (\cA^L_{\sigma, s} \times \cA^R_{\sigma, p})$.
\end{defi}

\begin{prop}\label{P:definition of antecedent}
Let $X$ be a shift space over $\cA$, $\sigma : \cA^* \to \cB^*$ a return morphism for the word $w$ and $Y = \sigma \cdot X$.
\begin{itemize}
\item
	If $u \in \cL(Y)$ does not contain any occurrence of $w$, then
	\[
		\cE_Y(u) = \cE_{\cL(\sigma)}(u),
	\]
	where we define
	\[
		\cL(\sigma) = \bigcup_{a \in \cA} \cL(\sigma(a)w).
	\]
\item
	If $u \in \cL(Y)$ contains an occurrence of $w$, then there exists a unique triplet $(s, v, p) \in \cB^* \times \cL(X) \times w\cB^*$ such that $u = s\sigma(v)p$ and for which the set $E_{X, s, p}(v)$ is not empty.
The bi-extensions of $u$ are then governed by those of $v$ through the equation
\begin{equation}\label{Eq: link between extensions}
	E_Y(u) =
	\{(a', b') \in \cB \times \cB \mid \exists (a,b) \in E_{X,s,p}(v) : \sigma(a) \in \cB^*a's \wedge \sigma(b)w \in pb'\cB^*\}.
\end{equation}
\end{itemize}
\end{prop}
\begin{proof}
Let us prove the first case.
For any $(b,c) \in E_Y(u)$, there exist a letter $a$ and a word $v \in \cL(X)$ such that $buc$ is a factor of $\sigma(a)\sigma(v)$. We can moreover assume that the first occurrence of $buc$ begins in $\sigma(a)$. Since $w$ is not a factor of $u$ and $\sigma(v)$ is prefix comparable with $w$, i.e. $\sigma(v)$ is either a prefix of $w$ or has $w$ as a prefix, we deduce that $buc$ is a factor of $\sigma(a)w$. This proves that $(b,c) \in E_{\cL(\sigma)}(u)$.
As every letter of $\cA$ appears in $X$, $\cL(\sigma)$ is included in $\cL(Y)$ and we conclude that $\cE_Y(u) = \cE_{\cL(\sigma)}(u)$.

For the second case, the proof is similar to the proof of~\citep[Proposition 4.1]{Gheeraert_Lejeune_Leroy:2021}.
\end{proof}

Under the assumptions of the previous result, the words $u$ that do not contain any occurrence of $w$ are the \emph{$\sigma$-initial factors} of $Y$. 
As the set of these factors and their extension graphs only depend on $\sigma$ and not on $X$ and $Y$, we will sometimes talk about \emph{$\sigma$-initial factors} without specifying $Y$. 
Note that a different convention for the choice of $w$ leads to a smaller set of $\sigma$-initial factors. However, by Proposition~\ref{P:return morphism for two words}, the set of right special $\sigma$-initial factors is the same.

Whenever $u$ and $v$ are as in the second case of the previous proposition, $v$ is called the \emph{antecedent} of $u$ under $\sigma$ and $u$ is said to be an \emph{extended image} of $v$.

Thus, any factor of $Y$ is either a $\sigma$-initial factor or an extended image of some factor in $X$. Using Proposition~\ref{P:definition of antecedent}, we directly know when the $\sigma$-initial factors of $Y$ are dendric. It only depends on $\sigma$ and not on $X$. We introduce the following definition.

\begin{defi}
A return morphism $\sigma$ is \emph{dendric} if every $\sigma$-initial factor is dendric in $\cL(\sigma)$.
\end{defi}

Note that, by Proposition~\ref{P:return morphism for two words} and because the proof of Proposition~\ref{P:definition of antecedent} works for any choice of $w$, the fact that a return morphism for $w$ is dendric is independent of the choice of $w$.

\begin{ex}\label{Ex:dendric return morphism}
The morphism $\beta : 0 \mapsto 0, 1 \mapsto 01, 2 \mapsto 02, 3 \mapsto 032$ is a return morphism for $0$. Its $\beta$-initial factors are $\eps$, $1$, $2$, $3$ and $32$. Note that the only one which is right special is $\eps$ and we can easily check that it is dendric. Thus, $\beta$ is a dendric return morphism.
\end{ex}

We will now use the second item of Proposition~\ref{P:definition of antecedent} to characterize the fact that all the extended images are dendric in $Y$. Equation~\ref{Eq: link between extensions} can also be seen in terms of extension graphs. To do that, we introduce two types of (partial) maps.

\begin{defi}
Let $\sigma : \cA^* \to \cB^*$ be a return morphism for the word $w$. For $s, p \in \cB^*$, we define the partial maps $\varphi^L_{\sigma, s}, \varphi^R_{\sigma, p} : \cA \to \cB$ as follows:
\begin{enumerate}
\item
	for $a \in \cA^L_{\sigma, s}$, $\varphi^L_{\sigma, s}(a)$ is the letter $a'$ such that $\sigma(a) \in \cB^*a's$;
\item
	for $b \in \cA^R_{\sigma, p}$, $\varphi^R_{\sigma, p}(b)$ is the letter $b'$ such that $\sigma(b)w \in pb'\cB^*$.
\end{enumerate}
\end{defi}

We denote by $\cE_{X, s, p}(v)$ the subgraph of $\cE_X(v)$ generated by the edges in $E_{X, s, p}(v)$. Remark that, with this definition, no vertex is isolated. In particular, the set of vertices might be strictly included in the disjoint union of $E_X^L(v) \cap \cA_{\sigma, s}^L$ and $E_X^R(v) \cap \cA_{\sigma, p}^R$.

The following proposition then directly follows from Equation~\ref{Eq: link between extensions}.

\begin{prop}\label{P:image of graphs by phi}
Let $X$ be a shift space over $\cA$, $\sigma : \cA^* \to \cB^*$ a return morphism for the word $w$, $Y = \sigma \cdot X$, and $u \in \cL(Y)$ containing an occurrence of $w$. If $(s, v, p)$ is the triplet given by Proposition~\ref{P:definition of antecedent}, then the extension graph of $u$ in $Y$ is the image of the graph $\cE_{X, s, p}(v)$ by the morphisms $\varphi^L_{\sigma, s}$ acting on the left vertices and $\varphi^R_{\sigma, p}$ acting on the right vertices.
\end{prop}

Remark that, as a consequence, $E^L_Y(u) \subseteq \varphi^L_{\sigma, s}(E^L_X(v))$ and $E^R_Y(u) \subseteq \varphi^R_{\sigma, p}(E^R_X(v))$. In particular, if $u$ is left (resp., right) special, $v$ is also left (resp., right) special and there exist two distinct letters $a, b \in \cA$ and two distinct letters $a', b' \in \cB$ such that
\[
	\sigma(a) \in \cB^*a's, \quad \sigma(b) \in \cB^*b's
\]
\[
	\text{(resp., } \sigma(a)w \in pa'\cB^*, \quad \sigma(b)w \in pb'\cB^* \text{)}.
\]

This observation motivates the following definition.

\begin{defi}
Let $\sigma : \cA^* \to \cB^*$ be a return morphism for the word $w$. For two distinct letters $a, b \in \cA$, we denote by $s_\sigma(a, b)$ the longest common suffix of $\sigma(a)$ and $\sigma(b)$ and by $p_\sigma(a, b)$ the longest common prefix of $\sigma(a)w$ and $\sigma(b)w$.
Finally, we define the two sets
\[
	\cT^L(\sigma) = \{s_\sigma(a,b) \mid a, b \in \cA, a \ne b\}
\]
and
\[
	\cT^R(\sigma) = \{p_\sigma(a,b) \mid a, b \in \cA, a \ne b\}.
\]
\end{defi}

The following result is~\citep[Proposition 4.8]{Gheeraert_Lejeune_Leroy:2021}.

\begin{prop}\label{P:extended images are dendric}
Let $X$ be a shift space over $\cA$, $\sigma : \cA^* \to \cB^*$ a return morphism, $Y = \sigma \cdot X$ and $v \in \cL(X)$ a dendric factor. All the extended images of $v$ under $\sigma$ are dendric in $Y$ if and only if the following conditions are satisfied
\begin{enumerate}
\item
	for every $s \in \cT^L(\sigma)$, $\cE_{X, s, \eps}(v)$ is connected;
\item
	for every $p \in \cT^R(\sigma)$, $\cE_{X, \eps, p}(v)$ is connected.
\end{enumerate}
\end{prop}

\begin{rem}\label{R:extended images when ordinary graph}
Recall that a word $v \in \cL(X)$ is ordinary if there exist $a, b \in \cA$ such that $E_X(v) = (E^L_X(v) \times \{b\}) \cup (\{a\} \times E^R_X(v))$. In this case, for all $s \in \cT^L(\sigma)$ (resp., $p \in \cT^R(\sigma)$), the graph $\cE_{X, s, \eps}(v)$ (resp., $\cE_{X, \eps, p}(v)$) is connected (or empty). Thus $v$ only has dendric extended images. In fact, one can see using Proposition~\ref{P:image of graphs by phi} that $v$ only has ordinary extended images.
\end{rem}

As a direct consequence of Propositions~\ref{P:definition of antecedent} and~\ref{P:extended images are dendric}, we can determine when the image of a dendric shift under a return morphism is dendric.

\begin{prop}\label{P:equiv dendric image of dendric shift}
Let $X$ be a dendric shift over $\cA$ and $\sigma : \cA^* \to \cB^*$ be a return morphism. The shift space $Y = \sigma \cdot X$ is dendric if and only if $\sigma$ is dendric and, for all $v \in \cL(X)$, for all $s \in \cT^L(\sigma)$ and for all $p \in \cT^R(\sigma)$, the graphs $\cE_{X, s, \eps}(v)$ and $\cE_{X, \eps, p}(v)$ are connected.
\end{prop}

In particular, by Remark~\ref{R:extended images when ordinary graph}, the image of an Arnoux-Rauzy shift under a return morphism is dendric if and only if the morphism is dendric.

In the case of an eventually dendric shift, we recall the following characterization.

\begin{prop}[Dolce and Perrin~\citep{Dolce_Perrin:2021}]\label{P:equiv eventually dendric DP}
Let $X$ be a shift space over $\cA$. The following properties are equivalent.
\begin{enumerate}
\item
	The shift $X$ is eventually dendric.
\item \label{item:equiv eventually dendric DP left}
	There exists $N \in \N$ such that for any left special word $v \in \cL_{\geq N}(X)$, there exists a unique letter $a \in \cA$ such that $va$ is left special. Moreover, $E^L_X(va) = E^L_X(v)$.
\item \label{item:equiv eventually dendric DP right}
	There exists $N \in \N$ such that for any right special word $v \in \cL_{\geq N}(X)$, there exists a unique letter $a \in \cA$ such that $av$ is right special. Moreover, $E^R_X(av) = E^R_X(v)$.
\end{enumerate}
\end{prop}

\begin{cor}
A shift space $X$ is eventually dendric if and only if any long enough bispecial factor is ordinary.
\end{cor}

Thus, the image of an eventually dendric shift under a return morphism is also an eventually dendric shift, as stated in the following proposition.

\begin{prop}\label{P:image of eventually dendric}
Let $X$ be an eventually dendric shift over $\cA$ and $\sigma : \cA^* \to \cB^*$ a return morphism for $w$. The shift space $Y = \sigma \cdot X$ is an eventually dendric shift. More precisely, if $N$ satisfies Proposition~\ref{P:equiv eventually dendric DP} for $X$, then $Y$ is eventually dendric of threshold at most $\max_{a \in \cA} |\sigma(a)| (N+1) + |w| - 1$.
\end{prop}
\begin{proof}
Let $K = \max_{a \in \cA} |\sigma(a)|$.
Every word $u$ in $\cL(Y)$ of length at least $K+|w|-1$ contains an occurrence of $w$ and thus, using Proposition~\ref{P:definition of antecedent}, has an antecedent $v \in \cL(X)$.
Furthermore, since $u$ is of the form $s \sigma(v) p$, if $u$ is of length at least $K (N+1) + |w| - 1$, then $|\sigma(v)| \geq K (N-1) + 1$, so the antecedent $v$ of $u$ is of length at least $N$. 
By hypothesis on $N$ and by Remark~\ref{R:extended images when ordinary graph}, $v$ then only has dendric extended images thus $u$ is dendric.
\end{proof}

\section{Other graphs characterizing dendricity}
\label{S:definition of G^L and G^R}

For any shift space, we introduce graphs which describe all the sets of left or right extensions for factors of length $n$ in the shift space.

\begin{defi}
Let $X$ be a shift space over $\cA$ and $n \geq 0$. The graph $G^L_n(X)$ (resp., $G^R_n(X)$) is the multi-graph with labeled edges such that
\begin{itemize}
\item
	its vertices are the elements of $\cA$,
\item
	for any $v \in \cL_n(X)$ and any distinct $a, b \in E^L_X(v)$ (resp., $a, b \in E^R_X(v)$) there is an (undirected) edge labeled by $v$ between the vertices $a$ and $b$. 
\end{itemize}
\end{defi}

\begin{ex}
Let $X$ be the Thue-Morse shift space over $\{0, 1\}$, i.e., the shift space generated by the substitution $\sigma: 0 \mapsto 01, 1 \mapsto 10$. The first few graphs $G^L_n(X)$ are given in Figure~\ref{F:exemple Thue-Morse}.
\begin{figure}
	\tikzset{node/.style={circle,draw,minimum size=0.5cm,inner sep=0pt}}
	\tikzset{title/.style={minimum size=0.5cm,inner sep=0pt}}

 	\begin{center}
	\begin{tikzpicture}
		\node[title](g0) {$G^L_0(X)$};
		\node[node](g00) [below left= 1.2cm and 0.3cm of g0] {$0$};
		\node[node](g01) [right= 1.5cm of g00] {$1$};
		\draw (g00) edge node[above,pos=.5] {$\varepsilon$} (g01);
		
		\node[title](g1) [right = 2.5cm of g0] {$G^L_1(X)$};
		\node[node](g10) [below left= 1.2cm and 0.3cm of g1] {$0$};
		\node[node](g11) [right= 1.5cm of g10] {$1$};
		\draw[-,bend right] (g10) edge node[below,pos=.5] {$1$} (g11);
		\draw[-,bend right] (g11) edge node[above,pos=.5] {$0$} (g10);
		
		\node[title](g2) [right = 2.5cm of g1] {$G^L_2(X)$};
		\node[node](g20) [below left= 1.2cm and 0.3cm of g2] {$0$};
		\node[node](g21) [right= 1.5cm of g20] {$1$};
		\draw[-,bend right] (g20) edge node[below,pos=.5] {$10$} (g21);
		\draw[-,bend right] (g21) edge node[above,pos=.5] {$01$} (g20);
		
		\node[title](g3) [right = 2.5cm of g2] {$G^L_3(X)$};
		\node[node](g30) [below left= 1.2cm and 0.3cm of g3] {$0$};
		\node[node](g31) [right= 1.5cm of g30] {$1$};
		\draw[-,bend right=10] (g30) edge node[below,pos=.5] {$100$} (g31);
		\draw[-,bend right=75] (g30) edge node[below,pos=.5] {$101$} (g31);
		\draw[-,bend right=75] (g31) edge node[above,pos=.5] {$010$} (g30);
		\draw[-,bend right=10] (g31) edge node[above,pos=.5] {$011$} (g30);
	\end{tikzpicture}
	\end{center}
	\caption{Graphs $G_n^L(X)$ for $n \in \{0, \dots, 3\}$.}
	\label{F:exemple Thue-Morse}
\end{figure}
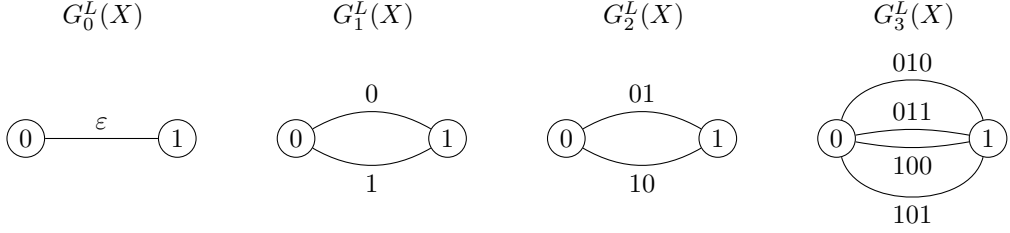
\end{ex}

\begin{rem}\label{R:basic properties of G_n}
By definition, the edges are only labeled by left (resp., right) special factors and the edges with a given label $v$ form a complete subgraph of $G^L_n(X)$ (resp., $G^R_n(X)$).
In addition, if $v = v_1 \cdots v_n$ labels an edge between $a$ and $b$ in $G^L_n(X)$ (resp., $G^R_n(X)$), then $v_1 \cdots v_k$ (resp., $v_{n-k+1} \cdots v_n$) labels an edge between $a$ and $b$ in $G^L_{k}(X)$ (resp., $G^R_{k}(X)$) for all $k \leq n$.
\end{rem}

We now explain the link between the properties of the extension graphs in $X$ and of the graphs $G^L_n(X)$, $G^R_n(X)$ regarding both acyclicity and connectedness.
We first need to define a notion of acyclicity for multi-graph with labeled edges. 

\begin{defi}
A multi-graph with labeled edges $G$ is \emph{acyclic for the labeling} if any simple cycle in $G$ only uses edges with the same label. 
\end{defi}

Recall that a cycle or a path is \emph{simple} if it uses distinct vertices (except for the first one and the last one in a cycle).

The following lemma is a direct consequence of the definition of the graphs $G^L_{n+1}(X)$ and $G^R_{n+1}(X)$.
\begin{lem} \label{L:relation between paths in extension graphs and G_n}
Let $X$ be a shift space and $v \in \cL_n(X)$.

The graph $\cE_X(v)$ contains the path $(a_1^L, b_1^R, a_2^L, \dots, b_k^R, a_{k+1}^L)$, with $a_i \ne a_{i+1}$, if and only if $G^L_{n+1}(X)$ contains the path
\[
    a_1 \overset{^{\mathlarger{vb_1}}}{\longdash} a_2 \dots \overset{^{\mathlarger{vb_k}}}{\longdash} a_{k+1}
\]

Symmetrically, the graph $\cE_X(v)$ contains the path $(a_1^R, b_1^L, a_2^R, \dots, b_k^L, a_{k+1}^R)$, with $a_i \ne a_{i+1}$, if and only if $G^R_{n+1}(X)$ contains the path
\[
    a_1 \overset{^{\mathlarger{b_1v}}}{\longdash} a_2 \dots \overset{^{\mathlarger{b_kv}}}{\longdash} a_{k+1}
\]
\end{lem}

\begin{prop}\label{P:equiv acyclic}
Let $X$ be a shift space and $N \geq 0$. The following properties are equivalent.
\begin{enumerate}
\item
	The graph $\cE_X(v)$ is acyclic for all $v \in \cL_{< N}(X)$.
\item
	The graph $G^L_n(X)$ (resp. $G^R_n(X)$) is acyclic for the labeling for all $n \leq N$.
\end{enumerate}
\end{prop}
\begin{proof}
Let us show the equivalence with the acyclicity of the graphs $G^L_n(X)$ by contraposition. 
The result with the graphs $G^R_n(X)$ is symmetric.
If $v \in \cL_n(X)$, $n < N$, is such that $\cE_X(v)$ contains a (non-trivial) cycle, then $G^L_{n+1}(X)$ contains a cycle by Lemma~\ref{L:relation between paths in extension graphs and G_n} and the edges of this cycle do not have the same label.
For the converse, assume that $G^L_n(X)$, $n \leq N$, contains a simple cycle
\[
    a_1 \overset{^{\mathlarger{u^{(1)}}}}{\longdash} a_2 \dots \overset{^{\mathlarger{u^{(k)}}}}{\longdash} a_1,
\]
with at least two distinct labels.
Let $v$ be the longest common prefix to all the $u^{(i)}$, $i \leq k$, and, for all $i \leq k$, let $b_i$ be the letter such that $vb_i$ is a prefix of $u^{(i)}$. 
By Remark~\ref{R:basic properties of G_n}, the graph $G^L_{|v| + 1}(X)$ thus contains the cycle
\[
    a_1 \overset{^{\mathlarger{vb_1}}}{\longdash} a_2 \dots \overset{^{\mathlarger{vb_k}}}{\longdash} a_{1}
\]
By definition of $v$, there exist $i < j$ such that $b_i \ne b_j$. We may assume that $b_1 \ne b_k$ because, if $b_1 = b_k$, since the subgraph of $G^L_{|v| + 1}(X)$ generated by the edges labeled by $vb_k$ is a complete graph, we can consider the path
\[
    a_2 \overset{^{\mathlarger{vb_2}}}{\longdash} a_3 \dots \overset{^{\mathlarger{vb_k}}}{\longdash} a_{2}
\]
instead.

By Lemma~\ref{L:relation between paths in extension graphs and G_n}, the graph $\cE_X(v)$ contains the cycle $(a_1^L,b_1^R,a_2^L,\dots,b_k^R,a_1^L)$ which is non-trivial as $b_1 \neq b_k$.
\end{proof}

We now look at the connectedness properties.

\begin{prop}\label{P:equiv connected with subgraphs}
Let $X$ be a shift, $N \geq 0$ and $C \subseteq \cA$. If the graph $\cE_X(v)$ is acyclic for all $v \in \cL_{<N}(X)$, then the following properties are equivalent.
\begin{enumerate}
\item
	For all $v \in \cL_{<N}(X)$ and all $a, b \in E^L_X(v) \cap C$, $a^L$ and $b^L$ are connected in $\cE_X(v)$ by a path avoiding vertices $c^L$, $c \notin C$.
\item
	For all $n \leq N$, the subgraph of $G^L_n(X)$ generated by the vertices in $C$ is connected.
\item
	The subgraph of $G^L_N(X)$ generated by the vertices in $C$ is connected.
\end{enumerate}
Similarly, the following are equivalent.
\begin{enumerate}
\item
	For all $v \in \cL_{<N}(X)$ and all $a, b \in E^R_X(v) \cap C$, $a^R$ and $b^R$ are connected in $\cE_X(v)$ by a path avoiding vertices $c^R$, $c \notin C$.
\item
	For all $n \leq N$, the subgraph of $G^R_n(X)$ generated by the vertices in $C$ is connected.
\item
	The subgraph of $G^R_N(X)$ generated by the vertices in $C$ is connected.
\end{enumerate}
\end{prop}

\begin{proof}
Let us show the first set of equivalences. We denote by $H_n$ the subgraph of $G^L_n(X)$ generated by the vertices in $C$.
Assume that the first property is satisfied. We prove that $H_n$ is connected by induction on $n$.
The graph $G^L_0(X)$ is a complete graph thus any subgraph is connected.
If $H_n$ is connected, then, to prove that $H_{n+1}$ is also connected, it suffices to show that any two vertices $a, b \in C$ that were connected by an edge in $H_n$ are connected by a path in $H_{n+1}$.
Let $a, b \in C$ be two such vertices and $v$ be the label of an edge between them. As $a$ and $b$ are left extensions of $v$, $a^L$ and $b^L$ are connected in $\cE_X(v)$ by a path avoiding vertices $c^L$, $c \notin C$. The vertices $a$ and $b$ are thus connected by a path in $H_{n+1}$ by Lemma~\ref{L:relation between paths in extension graphs and G_n}.
Note that the acyclic hypothesis is not needed for this implication.

For the converse, let us proceed by contraposition. Let $v \in \cL_n(X)$, $n < N$, be such that there exist two letters $a, b \in E^L_X(v) \cap C$ such that the vertices $a^L$ and $b^L$ are not connected in $\cE_X(v)$ by a path avoiding vertices $c^L$, $c \notin C$.
By definition, there is an edge labeled by $v$ between $a$ and $b$ in $G^L_n(X)$ thus, as this graph is acyclic for the labeling by Proposition~\ref{P:equiv acyclic}, any simple path between $a$ and $b$ in $G^L_n(X)$ only uses edges labeled by $v$.
By Remark~\ref{R:basic properties of G_n}, any simple path between $a$ and $b$ in $G^L_{n+1}(X)$ uses edges labeled by elements of $v\cA$. This is also true for the paths in $H_{n+1}$. However, using Lemma~\ref{L:relation between paths in extension graphs and G_n}, the existence of such path in $H_{n+1}$ would imply that $a^L$ and $b^L$ are connected in $\cE_X(v)$ by a path avoiding vertices $c^L$, $c \notin C$. Thus $H_{n+1}$ is not connected.

The equivalence between the second and the third properties follows from Remark~\ref{R:basic properties of G_n} as, with any path in $H_N$, we can associate a path in $H_n$ for $n \leq N$. 
\end{proof}

In particular, when $C = \cA$, Proposition~\ref{P:equiv connected with subgraphs} can be rewritten as follows.
\begin{prop}\label{P:equiv connected}
Let $X$ be a shift space and $N \geq 0$. If the graph $\cE_X(v)$ is acyclic for all $v \in \cL_{<N}(X)$, then the following properties are equivalent.
\begin{enumerate}
\item
	The graph $\cE_X(v)$ is connected for all $v \in \cL_{< N}(X)$.
\item
	The graph $G^L_n(X)$ (resp., $G^R_n(X)$) is connected for all $n \leq N$.
\item
	The graph $G^L_N(X)$ (resp., $G^R_N(X)$) is connected.
\end{enumerate}
\end{prop}

\begin{rem}\label{R:acyclic is necessary}
The result is false if we remove the acyclic hypothesis. Indeed, if a shift space $X$ is such that $\cL_3(X) = \{001, 010, 011, 100, 110, 111\}$, then the graph $\cE_X(0)$ is not connected. However, $G^L_2(X)$ is and, indeed, the graph $\cE_X(\varepsilon)$ contains a cycle (see Figure~\ref{F:acyclic hypothesis is necessary}) thus this does not contradict the previous result.
\begin{figure}
	\tikzset{node/.style={circle,draw,minimum size=0.5cm,inner sep=0pt}}
	\tikzset{title/.style={minimum size=0.5cm,inner sep=0pt}}

 	\begin{center}
	\begin{tikzpicture}
		\node[title](e0) {$\cE(0)$};
		\node[node](00l) [below left= 0.5cm and 0.5cm of e0] {$0$};
		\node[node](01l) [below= 0.7cm of 00l] {$1$};
		\node[node](01r) [right= 1.5cm of 00l] {$1$};
		\node[node](00r) [below= 0.7cm of 01r] {$0$};
		\path[draw,thick]
			(00l) edge node {} (01r)
			(01l) edge node {} (00r);
			
		\node[title](g2X) [right = 3cm of e0] {$G^L_2(X)$};
		\node[node](g0) [below left= 1cm and 0.3cm of g2X] {$0$};
		\node[node](g1) [right= 1.5cm of g0] {$1$};
		\draw[-,bend right] (g0) edge node[below,pos=.5] {$10$} (g1);
		\draw[-,bend right] (g1) edge node[above,pos=.5] {$11$} (g0);
			
		\node[title](ee) [right = 3cm of g2X] {$\cE(\varepsilon)$};
		\node[node](e0l) [below left= 0.5cm and 0.5cm of ee] {$0$};
		\node[node](e1l) [below= 0.7cm of e0l] {$1$};
		\node[node](e0r) [right= 1.5cm of e0l] {$0$};
		\node[node](e1r) [below= 0.7cm of e0r] {$1$};
		\path[draw,thick]
			(e0l) edge node {} (e0r)
			(e0l) edge node {} (e1r)
			(e1l) edge node {} (e0r)
			(e1l) edge node {} (e1r);
	\end{tikzpicture}
	\end{center}
	\caption{The acyclicity of $\cE_X(v)$ is necessary in Proposition~\ref{P:equiv connected}.}
	\label{F:acyclic hypothesis is necessary}
\end{figure}
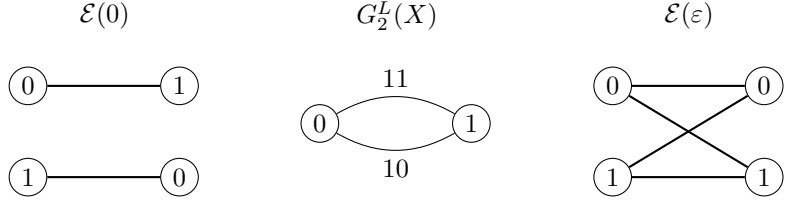
\end{rem}

As a direct consequence of Proposition~\ref{P:equiv acyclic} and Proposition~\ref{P:equiv connected}, we obtain the following characterization of dendric shifts.
\begin{cor}\label{C:equiv dendric graphs}
A shift space $X$ is dendric if and only if, for all $n \in \N$, the graph $G^L_n(X)$ (resp., $G^R_n(X)$) is acyclic for the labeling and connected.
\end{cor}

\section{Stabilization and eventually dendric shifts}
\label{S:stabilization for eventually dendric}

In the rest of the paper, we will be interested in whether two edges have the same label or not more than in the actual label of the edges. We will thus identify two multi-graphs $G$ and $G'$ with edges labeled by elements of $C$ and $C'$ respectively if they both have the same set of vertices and if there exists a bijection $\varphi : C \to C'$ such that there are $k$ edges labeled by $c \in C$ between $a$ and $b$ in $G$ if and only if there are $k$ edges labeled by $\varphi(c)$ between $a$ and $b$ in $G'$.

To make the distinction clearer, we will talk about multi-graphs with \emph{colored} edges instead of labeled edges.
Formally, a multi-graph with colored edges is a multi-graph with labeled edges, but we use the terminology ``colored'' to highlight the fact that we are not interested in the label but in the sets of edges having the same label.    
We naturally adapt the terminology of acyclicity for the labeling to acyclicity for the coloring.

Among the multi-graphs with colored edges, we will mainly be interested in graphs that can correspond to $G^L_n(X)$ (or $G^R_n(X)$) for some shift space $X$.
These graphs are exactly the ones that can be constructed as follows.

\begin{defi}\label{D:G(C_i)}
A multi-graph $G$ with colored edges and with set of vertices $V$ is a \emph{multi-clique} if there exist subsets $C_1, \dots, C_k$ of $V$ such that the set of edges is the union of the sets of colored edges $\left((C_i \times C_i) \setminus \diag(C_i)\right) \times \{c_i\}$, $i \leq k$, where $c_i$ is the color of the edges and $c_1, \dots, c_k$ are distinct colors. 
The multi-clique $G$ is then denoted $G(\{C_1, \dots, C_k\})$.
\end{defi}

\begin{ex}
If $C = \{0,1,2,3\}$, $C_1 = \{0,1\}$, $C_2 = \{1,2,3\}$ and $C_3 = \{0,3\}$, the multi-clique $G(\{C_1, C_2, C_3\})$ is represented in Figure~\ref{F:exemple graphe sous-ensembles}. It is not acyclic for the coloring because of the cycle $0,1,3,0$.

\begin{figure}
 	\begin{center}
	\begin{tikzpicture}
	
		\tikzstyle{every node}=[shape=circle, fill=none, draw=black,minimum size=20pt, inner sep=2pt]
		\node(0) at (0,0) {$0$};
		\node(1) at (2,0) {$1$};
		\node(2) at (2,-2) {$2$};
		\node(3) at (0,-2) {$3$};
		
		\draw[-,blue,thick] (0) to (1);
		\draw[-,red,thick] (1) to (2);
		\draw[-,red,thick] (1) to (3);
		\draw[-,red,thick] (3) to (2);
		\draw[-,thick] (0) to (3);
	\end{tikzpicture}
	\end{center}
	\caption{Multi-clique $G(\{\{0,1\}, \{1,2,3\}, \{0,3\}\})$}
	\label{F:exemple graphe sous-ensembles}
\end{figure}
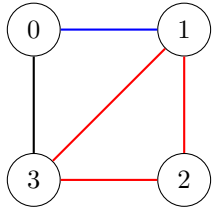
\end{ex}

With this definition, the colored version of $G^L_n(X)$ is the graph $G(\{E^L_X(v) : v \in \cL_n(X)\})$.
Remark that if $\Card(C_i) \leq 1$, then $G(\{C_1, \dots, C_k\}) = G(\{C_1, \dots, C_k\} \setminus \{C_i\})$ and, in particular, this shows once again that it suffices to consider the left-special factors to build $G^L_n(X)$.

In the case where the shift space $X$ is eventually dendric, the graphs $G^L_n(X)$ (resp., $G^R_n(X)$) are eventually constant, when seen as multi-graphs with colored edges. This is stated in the following result.

\begin{prop}\label{P:link eventually dendric with graphs}
Let $X$ be an eventually dendric shift. There exists $N \in \N$ such that, for all $n \geq N$,
\[
	G(\{E^L_X(v) \mid v \in \cL_n(X)\}) = G(\{E^L_X(v) \mid v \in \cL_N(X)\})
\]
\[
	\text{(resp., } G(\{E^R_X(v) \mid v \in \cL_n(X)\}) = G(\{E^R_X(v) \mid v \in \cL_N(X)\}\text{).}
\]
\end{prop}
\begin{proof}
It suffices to take $N$ satisfying condition~\ref{item:equiv eventually dendric DP left} (resp., condition~\ref{item:equiv eventually dendric DP right}) of Proposition~\ref{P:equiv eventually dendric DP}.
\end{proof}

The previous result is not an equivalence as there exist non eventually dendric shift spaces $X$ whose graphs $G^L_n(X)$ and $G^R_n(X)$ are eventually constant, when seen as multi-graphs with colored edges. This is the case of the following example.

\begin{ex}
The Chacon ternary shift space over $\{0,1,2\}$ is generated by the morphism $\sigma: 0 \mapsto 0012, 1 \mapsto 12, 2 \mapsto 012$.
It is of complexity $2n + 1$~\citep{Fogg} thus, by Proposition~\ref{P:complexity}, for each length $n$, there is either a unique left special word with three extensions or two left special words with two extensions each. The graph $G^L_n(X)$ then contains three or two edges respectively.
Since the graph $G^L_1(X)$ only contains two edges between $0$ and $2$ (labeled by $0$ and $1$), using Remark~\ref{R:basic properties of G_n}, it will be the case of $G^L_n(X)$ for all $n \geq 1$.
We can similarly show that the graphs $G^R_n(X)$ correspond to the same colored graph for all $n \geq 1$. However, the Chacon shift space is not eventually dendric~\citep{Dolce_Perrin:2021}.
\end{ex}

We can however show that, if $X$ is eventually connected, i.e. the extension graph of any long enough factor is connected, and if the graphs $G^L_n(X)$ or $G^R_n(X)$ are eventually constant, then $X$ is eventually dendric.

\begin{defi}
Let $X$ be an eventually dendric shift. The graph $G^L(X)$ (resp., $G^R(X)$) is the graph $G(\{E^L_X(v) \mid v \in \cL_N(X)\})$ (resp., $G(\{E^R_X(v) \mid v \in \cL_N(X)\})$) where $N$ is as in Proposition~\ref{P:link eventually dendric with graphs}.
\end{defi}

\begin{rem}
Another way to see the graph $G^L(X)$ is to consider the right infinite words and their left extensions. More precisely, for a right infinite word $x$, we denote by $E^L_X(x)$ the set of letters $a$ such that $\cL(ax) \subseteq \cL(X)$. We then have
\[
	G^L(X) = G(\{E^R_X(x) \mid x \in \cA^\N\}).
\]
This is also related to the notion of asymptotic equivalence. Indeed, each edge of $G^L(X)$ corresponds to a (right) asymptotic pair, i.e. two elements $x, y \in X$ such that $x_n = y_n$ for each $n \geq 0$ and $x_{-1} \ne y_{-1}$.
\end{rem}

By Corollary~\ref{C:equiv dendric graphs}, if $X$ is a dendric shift, then $G^L(X)$ and $G^R(X)$ are acyclic for the coloring and connected. 
However, the converse is not true, as can be seen in the following example.

\begin{ex}
\label{Ex:not dendric with acyclic graphs}
Let $\sigma$ be the morphism defined by $\sigma(0)= 0110$ and $\sigma(1) = 011$. 
It is a return morphism for $01$. Let $X$ be the image under $\sigma$ of a Sturmian shift. It is not a dendric shift space since the graph $\cE_X(\varepsilon)$ contains a cycle but, as the $\sigma$-initial factors are of length at most 4 and as every factor of a Sturmian shift space is ordinary, $X$ is eventually dendric with threshold at most 4 by Remark~\ref{R:extended images when ordinary graph}.
We can then show that the only words which are not dendric are $\varepsilon$ and $1$ thus $G^L(X) = G^L_2(X)$ and $G^R(X) = G^R_2(X)$. These graphs are acyclic for the coloring and connected since $01$ (resp., $10$) is the only left (resp., right) special word of length 2.
\end{ex}

\section{Action of morphisms on $G^L(X)$ and $G^R(X)$}
\label{S:image of graphs}

In this section, we use the graphs $G^L(X)$ and $G^R(X)$ to study the action of a return morphism on a dendric shift $X$.
We first show that these graphs contain all the information needed to know whether the image $Y$ of $X$ under a return morphism is dendric. This gives us a simpler formulation for Proposition~\ref{P:equiv dendric image of dendric shift}.
We then explain how to use these graphs to obtain $G^L(Y)$ and $G^R(Y)$.

\begin{defi}
Let $G$ be a multi-graph (with colored or uncolored edges) with set of vertices $\cA$ and $\sigma : \cA^* \to \cB^*$ be a return morphism for $w$. For all $s \in \cT^L(\sigma)$, we define the graph $G^L_{\sigma, s}$ as the subgraph of $G$ generated by the vertices in $\cA^L_{\sigma, s}$.
Similarly, for all $p \in \cT^R(\sigma)$, the graph $G^R_{\sigma, p}$ is the subgraph of $G$ generated by the vertices in $\cA^R_{\sigma, p}$.

If $X$ is an eventually dendric shift over $\cA$, we will write $G^L_{\sigma, s}(X)$ instead of $(G^L(X))^L_{\sigma, s}$ and $G^R_{\sigma, s}(X)$ instead of $(G^R(X))^R_{\sigma, s}$.
\end{defi}

\begin{prop}\label{P:equiv dendric image of dendric shift with graphs}
Let $X$ be a dendric shift space and $\sigma$ a return morphism. The image $Y$ of $X$ under $\sigma$ is dendric if and only if $\sigma$ is dendric and for all $s \in \cT^L(\sigma)$ and all $p \in \cT^R(\sigma)$, the graphs $G^L_{\sigma, s}(X)$ and $G^R_{\sigma, p}(X)$ are connected.
\end{prop}
\begin{proof}
As $G_n^L(X) = G^L(X)$ for all large enough $n$, the graphs $\cE_{X, s, \eps}(v)$ and $\cE_{X, \eps, p}(v)$ are connected for all $v \in \cL(X)$ if and only if $G^L_{\sigma, s}(X)$ and $G^R_{\sigma, p}(X)$ are connected. Indeed, it is a direct consequence of Proposition~\ref{P:equiv connected with subgraphs} with $C = \cA^L_{\sigma,s}$ and $C = \cA^R_{\sigma,p}$ respectively. The conclusion then follows from Proposition~\ref{P:equiv dendric image of dendric shift}.
\end{proof}

\begin{defi}
Let $G$ be the multi-clique $G(\{C_1, \dots, C_k\})$ with set of vertices $\cA$ and $\sigma : \cA^* \to \cB^*$ be a return morphism. The \emph{left image} of $G$ by $\sigma$ is the multi-clique
\[
	\sigma^L(G) = G(\{\varphi^L_{\sigma, s}(C_i) \mid i \leq k, s \in \cT^L(\sigma)\})
\]
where the set of vertices is $\cB$ and the \emph{right image} of $G$ by $\sigma$ is the multi-clique
\[
	\sigma^R(G) = G(\{\varphi^R_{\sigma, p}(C_i) \mid i \leq k, p \in \cT^R(\sigma)\})
\]
where the set of vertices is $\cB$.
\end{defi}

\begin{ex}\label{Ex:image of graph}
Let $\beta : 0 \mapsto 0, 1 \mapsto 01, 2 \mapsto 02, 3 \mapsto 032$ be the morphism of Example~\ref{Ex:dendric return morphism}. We have $\cT^L(\beta) = \{\eps, 2\}$ and the associated partial maps are given by
\[
    \varphi^L_{\beta, \eps} :
    \begin{cases}
        \;\, 0 &\!\!\!\mapsto 0\\
        \;\,1 &\!\!\!\mapsto 1\\
        2, 3 &\!\!\!\mapsto 2
    \end{cases}
    \qquad
    \text{and}
    \qquad
    \varphi^L_{\beta, 2} :
    \begin{cases}
        2 \mapsto 0\\
        3 \mapsto 3.
    \end{cases}
\]
If $G = G(\{\{0,1\}, \{1,2,3\}\})$, then we have
\begin{align*}
    \beta^L(G)
    &= G\left(\left\{\varphi^L_{\beta, \eps}(\{0,1\}), \varphi^L_{\beta, \eps}(\{1,2,3\}), \varphi^L_{\beta, 2}(\{0,1\}), \varphi^L_{\beta, 2}(\{1,2,3\})\right\}\right)\\
    &= G(\{\{0,1\}, \{1,2\}, \{0,3\}\}).
\end{align*}
The graphs are represented in Figure~\ref{F:left image of graph}.
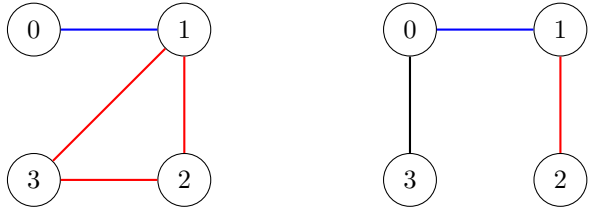
\begin{figure}
 	\begin{center}
	\begin{tikzpicture}
		\tikzset{node/.style={shape=circle, fill=none, draw=black,minimum size=20pt, inner sep=2pt}}
		\node[node](0) at (0,0) {$0$};
		\node[node](1) at (2,0) {$1$};
		\node[node](2) at (2,-2) {$2$};
		\node[node](3) at (0,-2) {$3$};
		
		\draw[-,blue,thick] (0) to (1);
		\draw[-,red,thick] (1) to (2);
		\draw[-,red,thick] (1) to (3);
		\draw[-,red,thick] (3) to (2);
		
		\node[node](0) at (5,0) {$0$};
		\node[node](1) at (7,0) {$1$};
		\node[node](2) at (7,-2) {$2$};
		\node[node](3) at (5,-2) {$3$};
		
		\draw[-,blue,thick] (0) to (1);
		\draw[-,red,thick] (1) to (2);
		\draw[-,thick] (0) to (3);
	\end{tikzpicture}
	\end{center}
	\caption{Graphs $G = G(\{\{0,1\},\{1,2,3\}\})$ (on the left) and $\beta^L(G)$ (on the right)}
	\label{F:left image of graph}
\end{figure}
\end{ex}

\begin{prop}\label{P:image of eventually dendric with graphs}
Let $X$ be an eventually dendric shift and $\sigma$ a return morphism for $w$. If $Y = \sigma \cdot X$, then $G^L(Y) = \sigma^L(G^L(X))$ and $G^R(Y) = \sigma^R(G^R(X))$.
\end{prop}
\begin{proof}
Let us prove the link between $G^L(X)$ and $G^L(Y)$.
By Proposition~\ref{P:image of eventually dendric}, $Y$ is also eventually dendric.
If we write $K = \max\{|\sigma(a)| : a \in \cA\}$,
let $N \in \N$ (resp., $M \in \N$) be large enough so that $N$ (resp., $M - K$) satisfies the condition~\ref{item:equiv eventually dendric DP left} of Proposition~\ref{P:equiv eventually dendric DP} for the shift $X$ (resp., $Y$).
Let us also assume that the antecedent of any word $u \in \cL(Y)$ of length at least $M$ is of length at least $N$.
By Proposition~\ref{P:link eventually dendric with graphs}, $G^L(X)$ (resp., $G^L(Y)$) is constructed with the left special factors of length $N$ (resp., $M$).

We define the application $f$ mapping any left special word $u \in \cL_M(Y)$ to the pair $(s_u, v) \in \cT^L(\sigma) \times \cL_N(X)$ where $(s_u, v_u, p_u)$ is the triplet associated with $u$ by Proposition~\ref{P:definition of antecedent} and $v$ is the prefix of length $N$ of $v_u$.

We will prove that $f$ is a bijection between the left special factors of length $M$ in $Y$ and the set
\[
	D = \{(s, v) \in \cT^L(\sigma) \times \cL_N(X) \mid \Card(\varphi^L_{\sigma, s}(E^L_X(v))) \geq 2\}.
\]
We will also show that, if $(s, v) = f(u)$, then
\[
	E^L_Y(u) = \varphi^L_{\sigma, s}(E^L_X(v))
\]
which will allow us to conclude since
\[
	\sigma^L(G^L(X)) = G(\{\varphi^L_{\sigma, s}(E^L_X(v)) \mid (s, v) \in D\}).
\]

Note that, if $f(u) = (s, v)$, then $s\sigma(v)w$ is a prefix of $u$, thus, by Proposition~\ref{P:definition of antecedent},
\[
	E^L_Y(u) \subseteq E^L_Y(s\sigma(v)w) = \varphi^L_{\sigma, s}(E^L_X(v))
\]
and $(s, v)$ is in $D$.

The application $f$ is injective. Indeed, if $f(u) = f(u') = (s, v)$, then $s_u = s = s_{u'}$ and $v$ is a prefix of both $v_u$ and $v_{u'}$. However, as $u$ and $u'$ are left special, $v_u$ and $v_{u'}$ must be left special. As they have a common prefix of length $N$ and $N$ is chosen large enough, it means that one is prefix of the other. Let us assume that $v_u$ is a prefix of $v_{u'}$. Then $s\sigma(v_u)w$ is a prefix of both $u$ and $u'$. Since $u$ is of length $M$, $s \sigma(v_u)w$ is a left special factor of length at least $M - K$ and, as $M$ is large enough, it is prefix of a unique left special factor of length $M$. This proves that $u = u'$ and that $f$ is injective.

We now prove the surjectivity. For any $(s, v) \in D$, by definition of $N$, there exists $v' \in v\cA^*$ left special such that $|s \sigma(v') w| \geq M$ and $E^L_X(v') = E^L_X(v)$.
If $u$ is the prefix of length $M$ of $s \sigma(v') w$, then $s_u = s$ and $v_u$ is a prefix of $v'$ of length at least $N$. Thus $s \sigma(v) w$ is a prefix of $u$. In addition, by Proposition~\ref{P:image of graphs by phi},
\[
	E^L_Y(u) \supseteq E^L_Y(s \sigma(v')w) = \varphi_{\sigma, s}^L(E^L_X(v')) = \varphi_{\sigma, s}^L(E^L_X(v)).
\]
In particular, $u$ is left special and $f(u) = (s, v)$. This proves that $f$ is surjective. Moreover, we have both inclusions thus $E^L_Y(u) = \varphi_{\sigma, s}^L(E^L_X(v))$.
\end{proof}

\section{$S$-adic characterization of minimal dendric shifts}
\label{S:S-adic characterization}

We recall the notion of $S$-adic representation and provide an $S$-adic characterization of both the dendric and the eventually dendric minimal shift spaces.

\subsection{$S$-adic representations}
\label{subsection:S-adic representations}
Let $\bsigma = (\sigma_n : \cA_{n+1}^* \to \cA_n^*)_{n \geq 0}$ be a sequence of morphisms such that 
$
\max_{a \in \cA_n} |\sigma_0 \circ \cdots \circ \sigma_{n-1}(a)|
$
goes to infinity when $n$ increases.
For $n < N$, we define the morphism $\sigma_{[n,N)} = \sigma_n \circ \sigma_{n+1} \circ \dots \circ \sigma_{N-1}$. The \emph{language $\cL^{(n)}({\bsigma})$ of level $n$ associated with $\bsigma$} is defined by 
\[
\cL^{(n)}({\bsigma}) = 
\left\{ w \in \cA_n^* \mid \mbox{$w$ occurs in $\sigma_{[n,N)}(a)$ for some $a \in \cA_N$ and $N>n$} \right\}.
\]

As $\max_{a \in \cA_N} |\sigma_{[n,N)}(a)|$ goes to infinity when $N$ increases, $\cL^{(n)}({\bsigma})$ defines a non-empty shift space $X_{\bsigma}^{(n)}$.
More precisely, $X_{\bsigma}^{(n)}$ is the set of points $x \in \cA_n^\mathbb{Z}$ such that $\cL (x) \subseteq \cL^{(n)}({\bsigma})$. 
Note that it may happen that $\cL(X_{\bsigma}^{(n)})$ is strictly contained in $\cL^{(n)}({\bsigma})$.
Also observe that for all $n$, $X_{\bsigma}^{(n)}$ is the image of $X_{\bsigma}^{(n+1)}$ under $\sigma_n$.

We set $X_{\bsigma} = X_{\bsigma}^{(0)}$ and call $\bsigma$ an {\em $S$-adic representation} of $X_{\bsigma}$. If we want to specify that the morphisms $\sigma_n$, $n \geq 0$, belong to a given family $\cS$, we say that $\bsigma$ is an $\cS$-adic representation.

We say that the sequence $\bsigma$ is {\em primitive} if, for any $n$, there exists $N>n$ such that for all $(a,b) \in \cA_n \times \cA_N$, $a$ occurs in $\sigma_{[n,N)}(b)$.
Observe that if $\bsigma$ is primitive, then $\min_{a \in \cA_n} |\sigma_{[0,n)}(a)|$ goes to infinity when $n$ increases, $\cL(X_{\bsigma}^{(n)})= \cL^{(n)}({\bsigma})$, and $X_{\bsigma}^{(n)}$ is a minimal shift space (see for instance~\citep[Lemma 7]{Durand:2000}). 

Given a minimal shift space $X$ over the alphabet $\cA$, we can build specific $S$-adic representations using return words.
Let $w \in \cL(X)$ be a non-empty word.
Recall that a return word to $w$ in $X$ is a non-empty word $r$ such that $rw \in \cL(X)$ and $rw$ contains exactly two occurrences of $w$, one as a prefix and one as a suffix.

We let $\cR_X(w)$ denote the set of return words to $w$ in $X$ and we omit the subscript $X$ whenever it is clear from the context.
The shift space $X$ being minimal, $\cR(w)$ is finite (see~\citep{Durand:1998}) thus we write $R_X(w) = \{1,\dots,\Card(\cR_X(w))\}$.
A morphism $\sigma: {R(w)}^* \to \cA^*$ is a {\em coding morphism} associated with $w$ if $\sigma(R(w)) = \cR(w)$.
It is trivially a return morphism for $w$.

Let us consider the set $\cD_w(X) = \{x \in {R(w)}^\Z \mid \sigma(x) \in X\}$.
It is a minimal shift space, called the {\em derived shift of $X$ (with respect to $w$)}.
We now show that derivation of minimal shift spaces allows us to build primitive $\cS$-adic representations with return morphisms.
We inductively define the sequences $(w_n)_{n \geq 0}$, $(R_n)_{n \geq 0}$, $(X_n)_{n \geq 0}$ and $(\sigma_n)_{n \geq 0}$ by
\begin{itemize}
\label{itemize S-adic}
\item
$X_0 = X$, $R_0 = \cA$ and $w_0 \in \cL(X) \setminus \{\varepsilon\}$;
\item
for all $n$, $R_{n+1} = R_{X_{n}}(w_n)$, $\sigma_n: R_{n+1}^* \to R_n^*$ is a coding morphism associated with $w_n$, $X_{n+1} = \cD_{w_n}(X_n)$ and $w_{n+1} \in \cL(X_{n+1}) \setminus \{\varepsilon\}$.
\end{itemize}

Observe that the sequence $(w_n)_{n \geq 0}$ is not uniquely defined, as well as the morphism $\sigma_n$ (even if $w_n$ is fixed).
However, to avoid heavy considerations when we deal with sequences of morphisms obtained in this way, we will speak about ``the'' sequence $(\sigma_n)_{n \geq 0}$ and it is understood that we may consider any such sequence. 

\begin{thm}[Durand~\citep{Durand:1998}]
\label{T:S-adic representation of minimal}
Let $X$ be a minimal shift space.
Using the notation defined above, the sequence of morphisms $\bsigma = (\sigma_n:R_{n+1}^* \to R_n^*)_{n \geq 0}$ is a primitive $S$-adic representation of $X$ using return morphisms.
In particular, for all $n$, we have $X_n = X_{\bsigma}^{(n)}$.
\end{thm}

Conversely, if $\bsigma$ is an $S$-adic representation of a minimal shift space $X$ and $\bsigma$ contains only return morphisms, then one can choose words $w_n$ such that $\bsigma$ is defined as above. Hence, it is primitive.

In the case of minimal dendric shifts, we have stronger properties for the $S$-adic representation $\bsigma$.
Recall that if $F_\cA$ is the free group generated by $\cA$, an automorphism $\alpha$ of $F_\cA$ is {\em tame} if it belongs to the monoid generated by the permutations of $\cA$ and by the {\em elementary automorphisms} 
\[
	\begin{cases}
		a \mapsto ab, \\
		c \mapsto c, & \text{for } c \neq a,
	\end{cases}
\qquad \text{and} \qquad
	\begin{cases}
		a \mapsto ba, \\
		c \mapsto c, & \text{for } c \neq a.
	\end{cases}
\]

\begin{thm}[Berthé et al.~\citep{bifix_decoding}]
\label{T:decoding dendric}
Let $X$ be a minimal dendric shift over the alphabet $\cA = \{1,\dots,d\}$.
For any $w \in \cL(X)$, $\cD_w(X)$ is a minimal dendric shift over $\cA$ and the coding morphism associated with $w$ is a tame automorphism of $F_\cA$.
As a consequence, if $\bsigma = (\sigma_n)_{n \geq 0}$ is the primitive $S$-adic representation from Theorem~\ref{T:S-adic representation of minimal}, then all morphisms $\sigma_n$ are tame dendric return morphisms.
\end{thm}

The previous result is true for any sequence $\bsigma$ build as above. However, some choices for the words $w_n$ give additional properties to the morphisms $\sigma_n$.

In particular, if $w_n$ is a letter, the morphism $\sigma_n$ is \emph{strongly left proper}, i.e. the image of each letter begins with the letter $w_n$ which does not appear elsewhere.
Moreover, if $w_n$ is a bispecial letter, then $w_n \in \cT^R(\sigma_n)$ and $\varepsilon \in \cT^L(\sigma_n)$. We can always chose such a $w_n$ as shown in the following results. The first one is a trivial lemma relying on the minimality of $X$.

\begin{lem}\label{L:subset of non left special letters}
Let $X$ be a minimal shift space over $\cA$. If there exist letters $a_1, \dots, a_k$ such that $a_{i+1}$ is the only right (resp., left) extension of $a_i$ for all $i < k$, and $a_1$ is the only right (resp., left) extension of $a_k$, then $\cA = \{a_1, \dots, a_k\}$.
In particular, $X$ does not have any right (resp., left) special letter.
\end{lem}

\begin{prop} 
Any minimal shift space $X$ such that the graph $\cE_X(\varepsilon)$ is connected has a bispecial letter.
\end{prop}
\begin{proof}
First, note that, because $\cE_X(\varepsilon)$ is connected, there is always a left and a right special letter.
Let $a_1$ be a right special letter. If it is bispecial, we can conclude. Otherwise, let $a_2$ be its unique left extension. Using Lemma~\ref{L:subset of non left special letters}, we know that $a_2 \ne a_1$.
In the graph $\cE_X(\varepsilon)$, $a_2^L$ is then the only neighbour of the vertex $a_1^R$ but this graph is connected thus the vertex $a_2^L$ is of degree at least two and $a_2$ is a right special letter.
Once again, if it is left special, we have found a bispecial letter and otherwise it has a unique left extension $a_3$. By Lemma~\ref{L:subset of non left special letters}, we must have $a_3 \ne a_1$ and $a_3 \ne a_2$.
We can iterate the process to define $a_4, a_5, \dots$. However, since the alphabet is finite, it stops at some point, meaning that we have found a bispecial letter.
\end{proof}

\subsection{Characterization via two graphs}

Similarly to what was done in the case of a ternary alphabet in~\citep{Gheeraert_Lejeune_Leroy:2021}, using the results of Section~\ref{S:image of graphs}, we can give an $S$-adic characterization of the minimal dendric shifts on an alphabet $\cA$. This characterization uses two graphs corresponding to the left and the right extensions respectively.

The edges of these graphs are given by the left (resp., right) valid triplets, as defined below. Note that we only assume that $\sigma$ is a return morphism and not a dendric return morphism, even if for the $S$-adic characterization (Theorem~\ref{T:main}), we will restrict ourselves to dendric return morphisms. 

\begin{defi}
\label{D:left valid triplet}
Let $\sigma : \cA^* \to \cB^*$ be a return morphism. The triplet $(G', \sigma, G)$ is \emph{left (resp., right) valid} if the following conditions are satisfied
\begin{enumerate}
\item
	$G$ is an acyclic for the coloring and connected multi-clique;
\item
	for all $s \in \cT^L(\sigma)$, $G^L_{\sigma, s}$ is connected (resp., for all $p \in \cT^R(\sigma)$, $G^R_{\sigma, p}$ is connected);
\item
	$G' = \sigma^L(G)$ (resp., $G' = \sigma^R(G)$) is an acyclic for the coloring and connected multi-clique.
\end{enumerate}
\end{defi}

Note that, if $s \in \cT^L(\sigma)$ (resp., $p \in \cT^R(\sigma)$) is of minimal length, then $G^L_{\sigma, s} = G$ (resp., $G^R_{\sigma, p} = G$) thus the second item of Definition~\ref{D:left valid triplet} implies that $G$ is connected.
Remark also that, if $X$ is dendric, then item 1 is satisfied for $(G^L(\sigma \cdot X), \sigma, G^L(X))$ and $(G^R(\sigma \cdot X), \sigma, G^R(X))$. Moreover, if $\sigma$ is a dendric return morphism, then $\sigma \cdot X$ is dendric if and only if item 2 is satisfied for both triplets by Proposition~\ref{P:equiv dendric image of dendric shift with graphs}, and in that case, item 3 is also satisfied by Proposition~\ref{P:image of eventually dendric with graphs}.

\begin{ex}
The morphism $\beta$ of Example~\ref{Ex:image of graph} is a dendric return morphism such that $\cT^L(\beta) = \{\eps, 2\}$ and $\cA^L_{\beta, \eps} = \{0,1,2,3\}$, $\cA^L_{\beta, 2} = \{2, 3\}$. Thus the triplet $(\beta^L(G), \beta, G)$, where $G = G(\{0,1\},\{1,2,3\})$ is the graph of Example~\ref{Ex:image of graph} is left valid since both graphs are acyclic for the coloring and connected multi-cliques and $G$ contains the edge $(2, 3)$
\end{ex}

We can deduce that a shift space is minimal dendric if and only if it has a primitive $S$-adic representation labeling an infinite path in the graphs built with the left (resp., right) valid triplets for dendric return morphisms.
The details are not given here as the idea is the same as in the stronger version that we will prove. Indeed, instead of considering all the valid triplets, we can restrict ourselves to the case where the graphs are \emph{colored trees}, i.e. all the edges have a different color and the corresponding uncolored graph is a tree.

Note that a colored tree is trivially acyclic for the coloring and connected. As we only work with multi-cliques, in what follows, any colored graph whose underlying graph is a tree is a colored tree. Conversely, an uncolored tree corresponds to a unique colored tree so we will make no distinction between the colored and the uncolored version.

For two multi-graphs $G$ and $G'$ with colored edges, we say that $G$ is a subgraph of $G'$ if the uncolored version of $G$ is a subgraph of the uncolored version of $G'$ (but the colors might not coincide).

\begin{lem}\label{L:breaking a component in two}
If the multi-clique $G = G(\{C_1, \dots, C_k\})$ is acyclic for the coloring and connected and if $D, E$ are such that
\[
	D \cup E = C_1 \quad \text{and} \quad \Card(D \cap E) = 1,
\]
then the multi-clique $G' = G(\{D, E, C_2, \dots, C_k\})$ is an acyclic for the coloring and connected subgraph of $G$.
\end{lem}
\begin{proof}
Let $\{c\} = D \cap E$. By construction, $G'$ is a subgraph of $G$ and any lost edge was between a vertex of $D \setminus \{c\}$ and a vertex of $E \setminus \{c\}$. Any two such vertices remain connected through $c$ thus $G'$ is connected. 
Moreover, any simple cycle of $G'$ corresponds to cycle of $G$ thus only uses edges corresponding to one of the $C_i$, $i \leq k$, in $G$. Since $c$ is the only vertex with both ingoing edges corresponding to $D$ and ingoing edges corresponding to $E$, any simple cycle using edges from $C_1 = D \cup E$ only uses edges corresponding to $D$ or edges corresponding to $E$. The conclusion follows.
\end{proof}

Recall that a covering tree of a graph $G$ is a subgraph $T$ of $G$ with the same vertices as $G$ and which is a tree. If $G$ is a colored graph, then $T$ does not need to share the same edges colors as $G$ and is always assumed to be a colored tree (or an uncolored tree, as explained above).

\begin{prop}\label{P:from graphs to trees}
Let $(G', \sigma, G)$ be a left (resp., right) valid triplet. For any covering tree $T'$ of $G'$, there exists a covering tree $T$ of $G$ such that $(T', \sigma, T)$ is left (resp., right) valid.
\end{prop}
\begin{proof}
We prove the result for left valid triplets.
Let $\mathcal{C}= \{C_1, \dots, C_k\}$ be such that $G = G(\mathcal{C})$.
Because $(G', \sigma, G)$ is left valid, one has $G' = \sigma^L(G) = G(\cC')$, where $\cC' = \{\varphi^L_{\sigma,s}(C_i) \mid i \leq k, s \in \cT^L(\sigma)\}$.
 
Assume that either $G$ or $G'$ is not a tree.
Observe that since $(G', \sigma, G)$ is a left valid triplet, $G$ and $G'$ are connected multi-cliques.
Thus if one of them is not a tree, it must contain a cycle. 
Whichever of $G$ and $G'$ is not a tree, we start by highlighting a word $s \in \mathcal{T}^L(\sigma)$, some set $C_i \in \mathcal{C}$ and two distinct vertices $a,b \in \varphi^L_{\sigma,s}(C_i)$ as follows.

\begin{itemize}
\item
	If $G'$ is not a tree, it contains a cycle. 
	Since $G' = G(\cC')$ is acyclic for the coloring, there exist $i \leq k$ and $s \in \cT^L(\sigma)$ such that $C = \varphi_{\sigma, s}^L(C_i)$ contains at least 3 elements. 
	The subgraph $T''$ of $T'$ generated by the vertices of $C$ is acyclic. Moreover, for any two vertices in $C$, any path connecting them in $G'$ uses edges corresponding to $C$ as $G'$ is acyclic for the coloring. Thus the path connecting them in $T'$ is in $T''$ and $T''$ is connected. 
	We choose $a \in C$ to be a vertex of degree 1 in $T''$ and let $b$ be its neighbor in $T''$.
	
\item
	If $G'$ is a tree but $G$ is not, then similarly to the first case, there exists $i \leq k$ such that $C_i$ contains at least 3 elements. 
	Let $s$ be the longest common suffix to all the $\sigma(d)$, $d \in C_i$; we have $s \in \cT^L(\sigma)$.
	The set $C = \varphi_{\sigma, s}^L(C_i)$ contains exactly two elements. 
	Indeed, it has at least two elements by definition of $s$ and, $G'$ being a tree, it cannot contain more than 2.
	Let us write $C = \{a,b \}$ with $b$ such that
	\[
		\Card(\{d \in C_i \mid \varphi_{\sigma, s}^L(d) = b\}) \geq 2.
	\]
\end{itemize}

In both cases, let $c \in C_i$ be such that $\varphi_{\sigma, s}^L(c) = b$ and let us note
\[
	D = \{c\} \cup \{d \in C_i \mid \varphi_{\sigma, s}^L(d) = a\}
\]
and
\[
	E = (C_i \setminus D) \cup \{c\}.
\]

We now use Lemma~\ref{L:breaking a component in two} to obtain a new left valid triplet $(H',\sigma,H)$, where 
\begin{itemize}
\item
$H$ is a connected subgraph of $G$;
\item
$T'$ is still a covering tree of $H'$;
\item
the total number of edges in $H$ and $H'$ is less than in $G$ and $G'$.
\end{itemize}
This will end the proof by iterating the construction.

If $\cD = \{D, E, C_1, \dots, C_k\} \setminus \{C_i\}$, then by Lemma~\ref{L:breaking a component in two}, the multi-clique $H = G(\cD)$ is an acyclic for the coloring and connected subgraph of $G$.

Consider now $\cD' = \{\varphi^L_{\sigma,s'}(F) \mid F \in \cD, s' \in \cT^L(\sigma)\}$ and $H' = G(\cD')$.
We show that the triplet $(H',\sigma,H)$ is left valid, i.e., we show item 2 and 3 of Definition~\ref{D:left valid triplet}. 
For any $s' \in \cT^L(\sigma)$, we now show that $H^L_{\sigma, s'}$ is connected. Moreover, we describe the sets $\varphi^L_{\sigma, s'}(D)$ and $\varphi^L_{\sigma, s'}(E)$.
Observe that for all $s'$, $H^L_{\sigma, s'}$ is a subgraph of $G^L_{\sigma, s'}$ obtained by removing the edges between $D\setminus \{c\}$ and $E\setminus\{c\}$.
\begin{itemize}
\item
If $s' = s$, then $c \in \cA^L_{\sigma, s'}$. 
The graph $H^L_{\sigma, s'}$ is connected as any pair of vertices is connected in $G^L_{\sigma, s'}$ and a path in $H^L_{\sigma, s'}$ can be deduced from the one in $G^L_{\sigma, s'}$ by replacing an edge $(d,e) \in (D\setminus\{c\}) \times (E \setminus \{c\})$ by the path $d,c,e$.
By definition of $D$ and $E$,
	\[
		\varphi^L_{\sigma, s'}(D) = \{a, b\} \quad \text{and} \quad \varphi^L_{\sigma, s'}(E) = \varphi^L_{\sigma, s'}(C_i) \setminus \{a\} = C \setminus \{a\}.
	\]

\item
If $s \in \cA^*ds'$ for some letter $d$, then $H^L_{\sigma, s'}$ is connected for the same reason and
	\[
		\varphi^L_{\sigma, s'}(D) = \{d\} \quad \text{and} \quad \varphi^L_{\sigma, s'}(E) = \varphi^L_{\sigma, s'}(C_i).
	\]

\item
If $s' \in \cA^*as$, then $\cA^L_{\sigma, s'} \cap E = \emptyset$ and $\cA^L_{\sigma, s'} \cap C_i = \cA^L_{\sigma, s'} \cap D$ thus no edge was lost from $G^L_{\sigma, s'}$ to $H^L_{\sigma, s'}$. In addition, 
	\[
		\varphi^L_{\sigma, s'}(D) = \varphi^L_{\sigma, s'}(C_i) \quad \text{and} \quad \varphi^L_{\sigma, s'}(E) = \emptyset.
	\]

\item
Otherwise, $as$ and $s'$ are not suffix comparable thus $\cA^L_{\sigma, s'} \cap C_i = \cA^L_{\sigma, s'} \cap E$, $\cA^L_{\sigma, s'} \cap D \subseteq \{c\}$ and no edge was lost from $G^L_{\sigma, s'}$ to $H^L_{\sigma, s'}$. We have 
	\[
		\varphi^L_{\sigma, s'}(D) = \varphi^L_{\sigma, s'}(\{c\}) 
		\quad \text{and} \quad 
		\varphi^L_{\sigma, s'}(E) = \varphi^L_{\sigma, s'}(C_i).
	\]
\end{itemize}
In particular, we have shown that
\[
    G(\cD') = G\left(\left(\cC' \setminus \{C\}\right) \cup \{\{a, b\}, C \setminus \{a\}\}\right)
\]
thus, using Lemma~\ref{L:breaking a component in two}, the graph
\[
    H' = G(\cD') = \sigma^L(H)
\]
is acyclic for the coloring and connected. This proves that the triplet $(H', \sigma, H)$ is left valid.

We finally show that $T'$ is still a subtree of $H'$ and that the total number of edges decreased.
By what have just seen, $H'$ is a subgraph of $G'$ obtained by removing the edges corresponding to $C$ between $a$ and the elements of $C \setminus \{a, b\}$. 
If $G'$ is not a tree, then $H'$ is a strict subgraph of $G'$ since $C$ contains at least three elements, and $a$ and $b$ where chosen so that $T'$ is a subtree of $H'$. 
If $G'$ is a tree, then $H' = G'$ and $H$ is a connected strict subgraph of $G$ as $D$ and $E$ both contain at least two elements. 
Thus we can iterate the process until both graphs are trees.
\end{proof}

We now define formally the graphs that we will use for the $S$-adic characterization.
\begin{defi}
Let $\cS$ be a set of return morphisms from $\cA^*$ to $\cA^*$. The graph $\cG^L(\cS)$ (resp., $\cG^R(\cS)$) is defined by
\begin{itemize}
\item
	each vertex corresponds to a tree whose vertices are the elements of $\cA$;
\item
	there is an edge from $T'$ to $T$ labeled by $\sigma \in \cS$ if $(T', \sigma, T)$ is a left (resp., right) valid triplet.
\end{itemize}
\end{defi}

The proof of the main result of this article uses the following lemma.

\begin{lem}\label{L:graph containing valid triplet}
Let $X$ be a dendric shift, $\sigma$ a dendric return morphism and $Y = \sigma \cdot X$. If there exist a left valid triplet $(G', \sigma, G)$ and a right valid triplet $(H',\sigma, H)$ such that $G$ is a covering subgraph of $G^L(X)$ and $H$ is a covering subgraph of $G^R(X)$, then $Y$ is dendric, $G'$ is a covering subgraph of $G^L(Y)$ and $H'$ is a covering subgraph of $G^R(Y)$.
\end{lem}
\begin{proof}
For all $s \in \cT^L(\sigma)$, the graph $G^L_{\sigma, s}$ is a covering subgraph of $G^L_{\sigma, s}(X)$. Since $(G', \sigma, G)$ is left valid, $G^L_{\sigma, s}$ is connected and so is $G^L_{\sigma, s}(X)$. The same reasoning proves that $G^R_{\sigma, p}(X)$ is connected for all $p \in \cT^R(\sigma)$. By Proposition~\ref{P:equiv dendric image of dendric shift with graphs}, $Y$ is dendric.

Let us prove that $G'$ is a covering subgraph of $G^L(Y)$. The proof for $H'$ and $G^R(Y)$ is similar.
First, note that $G^L(Y)$ and $G' = \sigma^L(G)$ have the same set of vertices which is the image alphabet of $\sigma$. Moreover, they are acyclic for the coloring and multi-clique thus they are simple. Hence, it suffices to prove that if there is an edge in $G'$, it is also an edge in $G^L(Y)$.

Let $(a, b)$ be an edge of $G'$ and let $G = G(\{C_1, \dots, C_k\})$ and $G^L(X) = G(\{C'_1, \dots, C'_l\})$.
By definition, $G' = \sigma^L(G)$ thus there exist $i \leq k$ and $s \in \cT^L(\sigma)$ such that $a, b \in \varphi^L_{\sigma, s}(C_i)$. In other words, there exist $c, d \in C_i$ such that $\varphi^L_{\sigma, s}(c) = a$ and $\varphi^L_{\sigma, s}(d) = b$. Since $(c, d)$ is an edge of $G$, it is an edge of $G^L(X)$ and there exists $j \leq l$ such that $c, d \in C'_j$. We then have $a, b \in \varphi^L_{\sigma, s}(C'_j)$ thus there is an edge between $a$ and $b$ in $\sigma^L(G^L(X))$, which is exactly $G^L(Y)$ by Proposition~\ref{P:image of eventually dendric with graphs}. We conclude that $G'$ is a covering subgraph of $G^L(Y)$.
\end{proof}

We can now prove the main result, that we recall here.

\mainThm*

\begin{proof}
Assume that $X$ is minimal dendric. The morphisms of $\cS$ being return morphisms, the sequence $\bsigma$ is as in Theorem~\ref{T:S-adic representation of minimal} and thus is primitive. 
Let us show that $\bsigma$ labels an infinite path in $\cG^L(\cS)$. Similarly, it will label an infinite path in $\cG^R(\cS)$.
For all $n \geq 0$, as $X_\bsigma^{(n+1)}$ and $X_\bsigma^{(n)}$ are dendric and since $X_\bsigma^{(n)}$ is the image of $X_\bsigma^{(n+1)}$ under $\sigma_n$, the triplet $(G^L(X_\bsigma^{(n)}), \sigma_n, G^L(X_\bsigma^{(n+1)}))$ is left valid by Proposition~\ref{P:equiv dendric image of dendric shift with graphs} and Proposition~\ref{P:image of eventually dendric with graphs}.
Let $T_0$ be a covering tree of $G^L(X) = G^L(X_\sigma^{(0)})$. By Proposition~\ref{P:from graphs to trees}, there exists a covering tree $T_1$ of $G^L(X_\sigma^{(1)})$ such that the triplet $(T_0, \sigma_0, T_1)$ is left valid thus $\sigma_0$ labels an edge from $T_0$ to $T_1$ in $\cG^L(\cS)$. We iterate the process to obtain an infinite path in $\cG^L(\cS)$.

Now assume that $\bsigma$ is primitive and labels a path $\left(G^L_n\right)_{n \geq 0}$ in $\cG^L(\cS)$ and a path $\left(G^R_n\right)_{n \geq 0}$ in $\cG^R(\cS)$. The shift $X_\bsigma$ is minimal by primitiveness of $\bsigma$~\citep{Durand:2000}.
Let $u \in \cL(X_\bsigma)$ be a factor. If $u$ is a $\sigma_0$-initial factor, then it is dendric since $\sigma_0$ is dendric. Otherwise, iterating Proposition~\ref{P:definition of antecedent}, there exist a unique $k > 0$ and a unique $v \in \cL(X_\bsigma^{(k)})$ $\sigma_k$-initial such that $u$ is an extended image of $v$ by $\sigma_0 \dots \sigma_{k-1}$.
Let $Y$ be an Arnoux-Rauzy shift and let $Z = \sigma_k \cdot Y$. Since $v$ is $\sigma_k$-initial, $\cE_Z(v) = \cE_{X_\bsigma^{(k)}}(v)$.
As $Y$ is an Arnoux-Rauzy shift, the graph $G^L(Y)$ (resp., $G^R(Y)$) is the complete (simple and monochromatic) graph on $\cA$, thus $G^L_{k+1}$ (resp., $G^R_{k+1}$) is a covering subgraph. By Lemma~\ref{L:graph containing valid triplet} applied $k+1$ times, we obtain that $\sigma_0 \dots \sigma_{k-1} \cdot Z$ is dendric thus, starting with the extension graph $\cE_Z(v)$, $v$ only has dendric extended images under $\sigma_0 \dots \sigma_{k-1}$ and $u$ is dendric.
\end{proof}

By construction, on an alphabet $\cA$ of size $d$, the graphs $\cG^L(\cS)$ and $\cG^R(\cS)$ have $d^{d - 2}$ vertices (see sequence \href{https://oeis.org/A000272}{A000272 on OEIS}) each, meaning that if we take the product of these graphs, we obtain a characterization with a unique graph $\cG(\cS)$ with $d^{2(d-2)}$ vertices.

However, it is possible to reduce the number of vertices with a clever use of permutations. Indeed, the vertices of $\cG(\cS)$ are pairs of labeled trees and some of these pairs are equivalent, in the sense that they are equal up to a permutation of the letters. Thus, by keeping only one pair per equivalence class and adding permutations to the morphisms of the set $\cS$, we obtain a smaller graph.

In the case of a ternary alphabet, this allows to replace a graph with 9 vertices by a graph with 2 vertices, as done in~\citep{Gheeraert_Lejeune_Leroy:2021}.
For an alphabet of size 4, we obtain a graph with 14 vertices instead of 256.

\subsection{An example: $\cS$-adic characterization of minimal dendric shifts with one right special factor of each length}

While we do not have an easy description of an explicit set $\cS$ for which any minimal dendric shift (over some fixed alphabet) has an $\cS$-adic representation, restricting ourselves to subfamilies of minimal dendric shifts allows to build examples.
Furthermore, and in accordance to what was stated in the beginning of the previous subsection, if we consider all possible multi-cliques as the vertices of $\cG^L(\cS)$ and $\cG^R(\cS)$ (and not just the covering trees), any primitive $\cS$-adic representation $\bsigma$ of $X$ will label, among others, the paths $(G^L(X_\bsigma^{(n)}))_{n \geq 0}$ and $(G^R(X_\sigma^{(n)}))_{n \geq 0}$ (in $\cG^L(\cS)$ and $\cG^R(\cS)$ respectively).

We now use this observation to build a graph characterizing the family $F$ of minimal dendric shift spaces over $\{0,1,2,3\}$ having exactly one right special factor of each length.

Let $X$ be such a shift space. If $\bsigma$ is an $S$-adic representation of $X$ (with dendric return morphisms) then, for all $n$, the shift space $X_\bsigma^{(n)}$ also has a unique right special factor of each length. Indeed, $G^R(X)$ is the complete monochromatic graph $K_4$ thus, by Proposition~\ref{P:image of eventually dendric with graphs}, $G^R(X_\bsigma^{(1)})$ contains a complete subgraph of size four and, as $X_\bsigma^{(1)}$ is dendric, we have $G^R(X_\bsigma^{(1)}) = K_4$. Hence, $X_\bsigma^{(1)}$ can only have one right special factor of each length. We iterate for $X_\bsigma^{(n)}$.

As a consequence, it makes sense to look for an $S$-adic characterization of $F$ using infinite paths.
It also implies that the elements of $F$ have $\cS_{1\text{R}}$-adic representations where $\cS_{1\text{R}}$ is the set of dendric return morphisms for a letter which has four right extensions. These morphisms can easily be found using the possible extension graphs for the empty word.

\begin{ex}
Assume that the extension graph of the empty word is given by
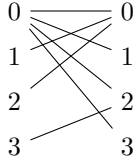
\begin{center}
\begin{tikzpicture}
\node (0L) at (0,0) {$0$};
\node (1L) [below of = 0L, node distance = .6cm] {$1$};
\node (2L) [below of = 1L, node distance = .6cm] {$2$};
\node (3L) [below of = 2L, node distance = .6cm] {$3$};

\node (0R) [right of = 0L, node distance = 1.5cm] {$0$};
\node (1R) [below of = 0R, node distance = .6cm] {$1$};
\node (2R) [below of = 1R, node distance = .6cm] {$2$};
\node (3R) [below of = 2R, node distance = .6cm] {$3$};

\draw (0L) -- (0R);
\draw (0L) -- (1R);
\draw (0L) -- (2R);
\draw (0L) -- (3R);

\draw (1L) -- (0R);
\draw (2L) -- (0R);
\draw (3L) -- (2R);
\end{tikzpicture}
\end{center}
The associated Rauzy graph of order 1, i.e., the graph where there is an edge from a letter $a$ to a letter $b$ whenever $ab \in \cL(X)$, is the following
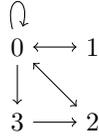
\begin{center}
\begin{tikzpicture}
\node (0) at (0,0) {$0$};
\node (1) [right of = 0, node distance = 1cm] {$1$};
\node (2) [below of = 1, node distance = 1cm] {$2$};
\node (3) [below of = 0, node distance = 1cm] {$3$};

\draw [->] (0) edge[loop above] (0);
\draw [<->] (0) edge (1);
\draw [<->] (0) edge (2);
\draw [->] (0) edge (3);
\draw [->] (3) edge (2);
\end{tikzpicture}
\end{center}
and the return words for $0$, corresponding to paths from and to the vertex $0$, are given by $0$, $01$, $02$ and $032$. Thus, the morphisms of $\cS$ associated with this extension graph are the elements of $\beta \Sigma_4$ where $\beta$ is the morphism of Example~\ref{Ex:image of graph}, i.e. $\beta : 0 \mapsto 0, 1 \mapsto 01, 2 \mapsto 02, 3 \mapsto 032$, and $\Sigma_4$ is the set of permutations over $\{0,1,2,3\}$.
\end{ex}

Up to a permutation of the alphabet, the possible extension graphs (with a unique right special letter) and their associated morphisms are given in Figure~\ref{F:morphisms with one right special}.
Note that, as we take the return morphisms for the only right special letter, there is, up to a permutation, only one morphism corresponding to each extension graph. The set $\cS_{1\text{R}}$ is then finite and given by $\Sigma_4 \{\alpha, \beta, \gamma, \delta\} \Sigma_4$, where $\alpha, \beta, \gamma, \delta$ are defined in Figure~\ref{F:morphisms with one right special}.

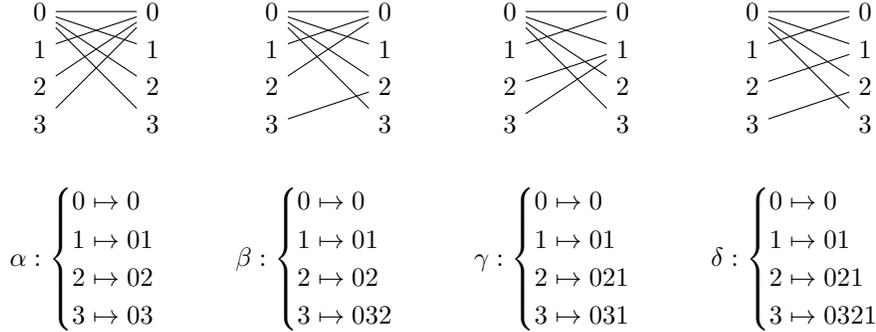
\begin{figure}[h]
\centering
\begin{tabular}{cccc}
\begin{tikzpicture}
\node (0L) at (0,0) {$0$};
\node (1L) [below of = 0L, node distance = .5cm] {$1$};
\node (2L) [below of = 1L, node distance = .5cm] {$2$};
\node (3L) [below of = 2L, node distance = .5cm] {$3$};

\node (0R) [right of = 0L, node distance = 1.5cm] {$0$};
\node (1R) [below of = 0R, node distance = .5cm] {$1$};
\node (2R) [below of = 1R, node distance = .5cm] {$2$};
\node (3R) [below of = 2R, node distance = .5cm] {$3$};

\draw (0L) -- (0R);
\draw (0L) -- (1R);
\draw (0L) -- (2R);
\draw (0L) -- (3R);
\draw (1L) -- (0R);
\draw (2L) -- (0R);
\draw (3L) -- (0R);
\end{tikzpicture}
&
\begin{tikzpicture}
\node (0L) at (0,0) {$0$};
\node (1L) [below of = 0L, node distance = .5cm] {$1$};
\node (2L) [below of = 1L, node distance = .5cm] {$2$};
\node (3L) [below of = 2L, node distance = .5cm] {$3$};

\node (0R) [right of = 0L, node distance = 1.5cm] {$0$};
\node (1R) [below of = 0R, node distance = .5cm] {$1$};
\node (2R) [below of = 1R, node distance = .5cm] {$2$};
\node (3R) [below of = 2R, node distance = .5cm] {$3$};

\draw (0L) -- (0R);
\draw (0L) -- (1R);
\draw (0L) -- (2R);
\draw (0L) -- (3R);
\draw (1L) -- (0R);
\draw (2L) -- (0R);
\draw (3L) -- (2R);
\end{tikzpicture}
&
\begin{tikzpicture}
\node (0L) at (0,0) {$0$};
\node (1L) [below of = 0L, node distance = .5cm] {$1$};
\node (2L) [below of = 1L, node distance = .5cm] {$2$};
\node (3L) [below of = 2L, node distance = .5cm] {$3$};

\node (0R) [right of = 0L, node distance = 1.5cm] {$0$};
\node (1R) [below of = 0R, node distance = .5cm] {$1$};
\node (2R) [below of = 1R, node distance = .5cm] {$2$};
\node (3R) [below of = 2R, node distance = .5cm] {$3$};

\draw (0L) -- (0R);
\draw (0L) -- (1R);
\draw (0L) -- (2R);
\draw (0L) -- (3R);
\draw (1L) -- (0R);
\draw (2L) -- (1R);
\draw (3L) -- (1R);
\end{tikzpicture}
&
\begin{tikzpicture}
\node (0L) at (0,0) {$0$};
\node (1L) [below of = 0L, node distance = .5cm] {$1$};
\node (2L) [below of = 1L, node distance = .5cm] {$2$};
\node (3L) [below of = 2L, node distance = .5cm] {$3$};

\node (0R) [right of = 0L, node distance = 1.5cm] {$0$};
\node (1R) [below of = 0R, node distance = .5cm] {$1$};
\node (2R) [below of = 1R, node distance = .5cm] {$2$};
\node (3R) [below of = 2R, node distance = .5cm] {$3$};

\draw (0L) -- (0R);
\draw (0L) -- (1R);
\draw (0L) -- (2R);
\draw (0L) -- (3R);
\draw (1L) -- (0R);
\draw (2L) -- (1R);
\draw (3L) -- (2R);

\end{tikzpicture}
\\ \\
\begin{tabular}[t]{l}
$\alpha:
\begin{cases}
    0 \mapsto 0     \\
	1 \mapsto 01 	\\ 
	2 \mapsto 02	\\
	3 \mapsto 03
\end{cases}
$
\end{tabular}
&
\begin{tabular}[t]{l}
$\beta:
\begin{cases}
    0 \mapsto 0     \\
	1 \mapsto 01 	\\ 
	2 \mapsto 02	\\
	3 \mapsto 032
\end{cases}
$
\end{tabular}
&
\begin{tabular}[t]{l}
$\gamma:
\begin{cases}
    0 \mapsto 0     \\
	1 \mapsto 01 	\\ 
	2 \mapsto 021	\\
	3 \mapsto 031
\end{cases}
$
\end{tabular}
&
\begin{tabular}[t]{l}
$\delta:
\begin{cases}
    0 \mapsto 0     \\
	1 \mapsto 01 	\\ 
	2 \mapsto 021	\\
	3 \mapsto 0321
\end{cases}
$
\end{tabular}
\end{tabular}
\caption{Possible extension graphs of the empty word in a minimal dendric shift over $\{0,1,2,3\}$ and their associated return morphisms.}
\label{F:morphisms with one right special}
\end{figure}

We now build the graph that will give us the characterization.
For the right side, we consider that the graph $\cG^R(\cS_{1\text{R}})$ is built with all the right valid triplets. However, as we have observed before, any $\cS_{1\text{R}}$-adic representation $\bsigma$ of $X \in F$ labels the path $(G^R(X_\bsigma^{(n)}))_{n \geq 0} = (K_4)_{n \geq 0}$. We can thus assume that the right side graph is reduced to the vertex $K_4$ without losing $\cS_{1\text{R}}$-adic representations.

For the left side, we consider the usual graph $\cG^L(\cS_{1\text{R}})$ (with trees as vertices). When taking the product with the right side graph, we can use permutations and, since $K_4$ is symmetric, this gives us a graph with only two vertices $(G_1, K_4)$ and $(G_2, K_4)$ where $G_1$ and $G_2$ are represented in Figure~\ref{F:graphs with fixed permutations}.

\begin{figure}
\centering
\begin{tikzpicture}
\node (0) at (0,0) {$0$};
\node (1) [above of = 0, node distance = 1cm] {$1$};
\node (2) [below left of = 0, node distance = 1cm] {$2$};
\node (3) [below right of = 0, node distance = 1cm] {$3$};

\draw [-] (0) edge (1);
\draw [-] (0) edge (2);
\draw [-] (0) edge (3);

\node (02) [right of = 0, node distance = 3cm] {$1$};
\node (12) [right of = 02, node distance = 1cm] {$0$};
\node (22) [right of = 12, node distance = 1cm] {$2$};
\node (32) [right of = 22, node distance = 1cm] {$3$};

\draw [-] (02) edge (12);
\draw [-] (12) edge (22);
\draw [-] (22) edge (32);
\end{tikzpicture}
\caption{Graphs $G_1$ on the left and $G_2$ on the right}
\label{F:graphs with fixed permutations}
\end{figure}
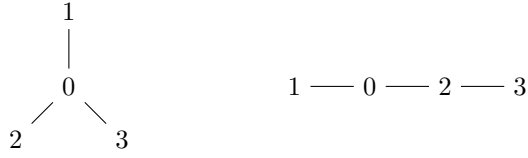

In the end, we obtain that each minimal shift space in $F$ has a primitive $\cS_{1\text{R}}$-adic representation in the graph represented in Figure~\ref{F:characterization of unique right special}. 
Note that, to reduce the number of morphisms, we added a loop on $(G_1, K_4)$ (resp., $(G_2, K_4)$) labeled with the permutations that fix the graph $G_1$ (resp., $G_2$). We will then consider that $\cS_{1\text{R}}$ also contains the permutations. To improve the readability of the graph, we also introduce the following notations: $\Sigma_{\{1,2,3\}}$ is the subset of permutations of $\Sigma_4$ fixing $0$, $\pi_{ab}$ stands for the cyclic permutation $(a\ \ b)$ and we define the following morphisms
\[
    \beta' = \beta \pi_{23}, \quad \delta' = \delta \pi_{23}, \quad \gamma' = \gamma \pi_{12}, \quad \gamma'' = \pi_{23} \gamma \pi_{23}.
\]

\begin{figure}[h]
\begin{center}
\scalebox{.9}{
\begin{tikzpicture}[scale=.8]
\node (1) at (0,0) {$(G_1, K_4)$};
\node (2) at (4,0) {$(G_2, K_4)$};
\node (fake1) at (-1,2) {};
\node (fake2) at (5,2) {};

\path (fake1) edge [->] node [pos=0.5,left] {$\Sigma_4$} (1);
\path (fake2) edge [->] node [pos=0.5,right] {$\Sigma_4$} (2);

\path (1) edge [loop left, ->] node [pos=0.5,left,align=center] {$\Sigma_{\{1,2,3\}}$, \\
$\alpha$, $\pi_{01}\alpha\pi_{01}$, $\pi_{02}\alpha\pi_{02}$, $\pi_{03}\alpha\pi_{03}$,\\
$\gamma'\pi_{02}$,\\
$\delta\pi_{02}$, $\delta'\pi_{02}$} (1);

\path (2) edge [loop right, ->] node [pos=0.5,right,align=center] {
$\pi_{02}\pi_{13}$,\\
$\alpha$, $\pi_{01}\alpha\pi_{01}$, $\pi_{02}\alpha\pi_{02}$, $\pi_{03}\alpha\pi_{03}$,\\
$\pi_{01}\pi_{02}\beta\pi_{01}$, $\pi_{01}\pi_{02}\beta'\pi_{01}$,\\ $\pi_{02}\beta\pi_{03}$, $\pi_{02}\beta'\pi_{03}$,\\
$\gamma\pi_{01}$,$\gamma\pi_{01}\pi_{03}$\\
$\gamma''\pi_{01}$,$\gamma''\pi_{01}\pi_{03}$} (2);

\path (1) edge [bend left, ->] node [pos=0.5,above,align=center] {
$\beta$, $\beta'$,
$\gamma'\pi_{01}$,$\gamma'\pi_{13}\pi_{01}$\\
$\delta\pi_{01}$, $\delta'\pi_{01}$,\\ $\delta\pi_{03}$, $\delta'\pi_{03}$} (2);

\path (2) edge [bend left, ->] node [pos=0.5,below,align=center] {
$\pi_{02}\beta\pi_{02}$, $\pi_{02}\beta'\pi_{02}$,\\
$\gamma \pi_{02}$, $\gamma'' \pi_{02}$} (1);
\end{tikzpicture}
}
\end{center}
\caption{A shift space over $\{0,1,2,3\}$ is minimal, dendric and has exactly one right special factor of each length if and only if it has a primitive $\cS_{1\text{R}}$-adic representation labeling an infinite path in this graph.}
\label{F:characterization of unique right special}
\end{figure}
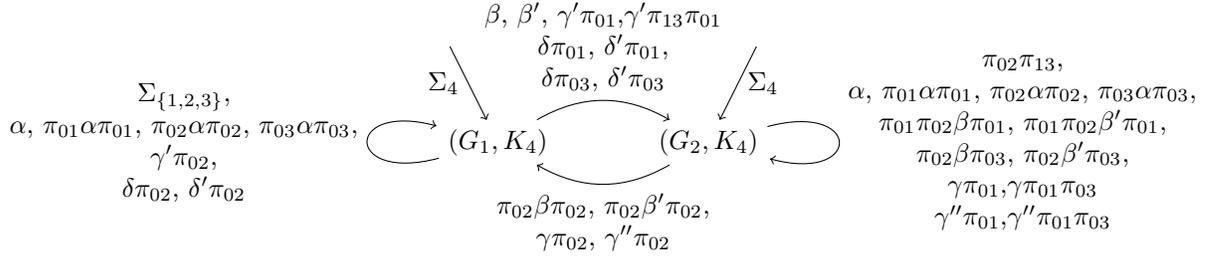

The converse result is also true, i.e. if a shift space $X$ has a primitive $\cS_{1\text{R}}$-adic representation $\bsigma$ labeling an infinite path in the graph of Figure~\ref{F:characterization of unique right special}, then it is a minimal dendric shift space over $\{0,1,2,3\}$ having exactly one right special factor of each length. Indeed, it is minimal dendric by the proof of  Theorem~\ref{T:main}. To conclude, we show that we can find infinitely many factors having four right extensions in $X$. This will imply that, for each length, there is a factor having four right extensions and it must be the unique right special factor since $X$ is dendric.

It follows from this observation: if $v$ has a left extension $a$ such that $a v$ has four right extensions and if $\sigma \in \cS_{1\text{R}}$ is a return morphism for $\ell$, then the extended image $\sigma(v) \ell$ has a left extension $b$ such that $b \sigma(v) \ell$ has four right extensions. Indeed, it suffices to take $b = \varphi^L_{\sigma, \eps}(a)$ by Proposition~\ref{P:image of graphs by phi} and by definition of the morphisms of $\cS_{1\text{R}}$.
For any $k > 0$, starting with $\eps$ in $X_\bsigma^{(k)}$, we can then find a word $u \in \cL(X)$ of length at least $k$ having four right extensions.

\subsection{Characterization of eventually dendric shifts}

Using the $S$-adic characterization obtained for minimal dendric shifts, we can deduce an $S$-adic characterization of minimal eventually dendric shifts. It is based on the fact that any eventually dendric shift space can be obtained from a dendric shift space in the following way.

\begin{prop}
Let $X$ be a minimal eventually dendric shift space with threshold $N$. For any non empty $w \in \cL(X)$, the derived shift $\cD_w(X)$ is eventually dendric with threshold at most $\max\{0, N - |w|\}$.
\end{prop}
\begin{proof}
The proof is almost exactly the same as in the dendric case (see~\citep[Theorem 5.13]{bifix_decoding}), replacing~\citep[Proposition 5.7]{bifix_decoding} by~\citep[Lemma 9.3]{Dolce_Perrin:2021}.
\end{proof}

By~\citep[Theorem 7.3]{Dolce_Perrin:2021}, if $X$ is eventually dendric and $w$ is of length at least $N$, $\cD_w(X)$ is on an alphabet of size $1 + p_X(N+1) - p_X(N)$.

The next result directly follows from the previous one and from Proposition~\ref{P:image of eventually dendric}.

\begin{thm}
\label{T:ult dendric iff derived dendric}
A minimal shift $X$ is eventually dendric if and only if it has a dendric derived shift, or equivalently, if and only if every derived shift with respect to a long enough factor of $X$ is dendric.
\end{thm}

The $S$-adic characterization of minimal eventually dendric shifts is now a direct consequence of Theorem~\ref{T:main}.
\begin{thm}
\label{thm:caract ev dendric}
A shift space is minimal eventually dendric if and only if it has a primitive $S$-adic representation $(\sigma_n)_{n \geq 0}$ where $\sigma_0$ is a return morphism, $\sigma_n$ is a dendric return morphism for all $n \geq 1$ and $(\sigma_n)_{n \geq 1}$ labels infinite paths in $\cG^L(\cS)$ and $\cG^R(\cS)$ where $\cS = \{\sigma_n : n \geq 1\}$.
\end{thm}

\subsection{Decidability of eventual dendricity for minimal substitutive shifts}

A deeply studied class of shift spaces is the one of substitutive (or morphic) shifts that roughly correspond to $S$-adic spaces with an eventually periodic $S$-adic representation.
A shift space $X$ is said to be morphic if there exist morphisms $\sigma:\cB^* \to \cB^*$ and $\tau:\cB^* \to \cA^*$ such that 
\[
X = \{x \in \cA^\mathbb{Z} \mid \forall u \in \cL(x), \exists b \in \cB, n \in \mathbb{N}: u \in \cL(\tau(\sigma^n(b))\}.
\]
We let $X(\sigma,\tau)$ denote such a shift.
In this section, we show that, under minimality, (eventual) dendricity is decidable for this class of shift spaces.
Whenever $\tau$ is the identity, the minimality of $X$ is also decidable~\citep{Perrin_Beal_Restovo}.

It is classical~\citep{Cobham:erasing} that we can suppose that $\sigma$ is non-erasing ($\sigma(a) \neq \varepsilon$ for every letter) and $\tau$ is a coding ($|\tau(a)|=1$ for every letter) so we always make such an assumption.

Observe that one usually considers the restricted version of morphic shifts consisting in shifts generated by a morphic word.
The shift space generated by an infinite word $x \in \cA^\mathbb{N}$ is the shift space $X_x = \{y \in \cA^\mathbb{Z} \mid \cL(y) \subseteq \cL(x)\}$.
Whenever $x$ is {\em uniformly recurrent} (that is, every finite word $u \in \cL(x)$ occurs infinitely many times in $x$ and with bounded gaps), the shift space $X_x$ is minimal. 
An infinite word $x$ is {\em morphic} if there exist morphisms $\sigma:\cB^* \to \cB^*$ and $\tau:\cB^* \to \cA^*$ such that $\sigma$ is {\em prolongable} on some letter $b \in \cB$ (that is, $\sigma(b) \in b\cB^*$ and $\lim_{n \to +\infty}|\sigma^n(b)|=+\infty$) and $x = \tau(\sigma^\omega(b)) = \lim_{n\to +\infty} \tau(\sigma^n(b^\omega))$.
It is also decidable whether a morphic word is uniformly recurrent~\citep{Durand:UR}. Though, it is not a requirement for it to generate a minimal shift space, i.e., a non uniformly recurrent word $x$ could generate a minimal shift space.
Thus, deciding if $X(\sigma,\tau)$ is minimal is an open problem.

A key argument to decide eventual dendricity in a minimal morphic shift is to make more deterministic the $S$-adic construction by return words described in Section~\ref{subsection:S-adic representations}.
More precisely, using the notation of the $S$-adic construction of Section~\ref{subsection:S-adic representations} (on Page~\pageref{itemize S-adic}), we want to algorithmically chose the words $w_n$ for which we consider return words as well as the coding morphisms $\sigma_n:R_{n+1}^* \to R_n^*$.
This can be achieved by fixing some morphic word generating $X(\sigma,\tau)$. 

\begin{lem}
\label{lem:compute morphisms}
If $X(\sigma,\tau)$ is minimal, then it is generated by a uniformly recurrent morphic word $x$ and one can compute morphisms $\sigma',\tau'$ such that $x = \tau'(\sigma'^\omega(b'))$.
\end{lem}
\begin{proof}
If $X(\sigma,\tau)$ is finite, then it is generated by $x = u^\omega$, where $u$ is a finite word. It is direct to check that $x = \sigma'^\omega(b')$, where $b'$ is the first letter of $u$ and $\sigma'(a) = u^2$ for every letter $a$.

Assume now that $X(\sigma,\tau)$ is aperiodic and let $B_g, B_b$ be the subalphabets of growing and bounded letters, i.e.,
\begin{align*}
    B_g &= \{b \in \cB \mid \lim_{n\to +\infty} |\sigma^n(b)| = +\infty\} \\
    B_b &= \{b \in \cB \mid (|\sigma^n(b)|)_{n \geq 0} \text{ is bounded}\}. 
\end{align*}
The set $B_g$ is non-empty by aperiodicity, so let $a \in B_g$.
For every $n \in \mathbb{N}$, we set $\sigma^n(a) = u_n a_n v_n$, where $a_n \in B_g$ and $u_n \in B_b^*$.
By the pigeonhole principle, we can find $k<\ell$ such that $a_k = a_\ell = b'$.
We thus have $\sigma^{\ell-k}(b') =ub'v$ for some word $u \in B_b^*$, and $\sigma^{n(\ell-k)}(b')$ starts with 
\[
\sigma^{(n-1)(\ell-k)}(u)\sigma^{(n-2)(\ell-k)}(u) \cdots \sigma^{\ell-k}(u)u
\]
for all $n$.
As $u \in B_b^*$, the sequence $(\sigma^{n(\ell-k)}(u))_{n \geq 0}$ is eventually periodic. Hence, as $\sigma$ and $\tau$ are non-erasing, $u = \varepsilon$ otherwise this contradicts the aperiodicity of $X(\sigma,\tau)$.
In other words, $\sigma' = \sigma^{\ell-k}$ is prolongable on $b'$ and, by minimality, the word $\tau(\sigma'^\omega(b'))$ generates $X(\sigma,\tau)$.
\end{proof}

Assume now that $x \in \cA^\mathbb{N}$ is a uniformly recurrent morphic word generating $X(\sigma,\tau)$.
If $u$ is a non-empty prefix of $x$, we set $R(u) = \{1,\dots,\Card(\cR(u))\}$ and we define $\theta_{u}:R(u)^* \to \cA^*$ so that $\theta_{u}(R(u)) = \cR(u)$ and according to the order of appearance of return words to $u$ in $x$. 
That is, for all $i \in R(u)$, if $k$ is the first occurrence of $\theta_{u}(i)u$ in $x$, then $x_{[0,k)}$ belongs to $\theta_{u}(\{1,\dots,i-1\}^*)$.
Then there exists a unique infinite word $\cD_u(x)$, called the {\em derived sequence of $x$ (with respect to $u$)} such that $\theta_{u}(\cD_u(x)) = x$.

\begin{thm}[Durand~\citep{Durand:UR}]
\label{thm:durand_caclculalbe}
Let $x = \tau(\sigma^\omega(b))$ be an aperiodic uniformly recurrent morphic word in $\cA^\mathbb{N}$.
\begin{enumerate}
\item
For every non-empty prefix $u$ of $x$, the morphism $\theta_{u}$ is computable and there exist some computable morphisms $\sigma_{u}: \cC^* \to \cC^*$ and $\tau_{u}: \cC^* \to R(u)^*$ such that $\cD_u(x) = \tau_{u}(\sigma_{u}^\omega(1))$.
\item
There is computable constant $D$ such that the set $\{(\sigma_{u},\tau_{u}) \mid u \text{ non-empty prefix of } x\}$ has cardinality at most $D$. 
In particular, the number of derived sequences of $x$ (on its non-empty prefixes) is at most $D$. 
\end{enumerate}
\end{thm}
\begin{proof}
The existence of $\sigma_{u}$ and $\tau_{u}$ is~\citep[Proposition 28]{Durand:UR} and the fact that $\theta_{u}$, $\sigma_{u}$ and $\tau_{u}$ are computable is explained in~\citep[Section 4]{Durand:UR}. 
The bound on the number of possible pairs $(\sigma_{u},\tau_{u})$ is~\citep[Theorem 29]{Durand:UR}.
\end{proof}

\begin{lem}[Durand~\citep{Durand:1998}]
\label{lemma:derived of derived}
Let $x$ be an aperiodic uniformly recurrent infinite word and let $u$ be a non-empty prefix of $x$.
For every non-empty prefix $v$ of $\cD_u(x)$, we have $\cD_v(\cD_u(x)) = \cD_w(x)$, where $w = \theta_{x,u}(v)u$\footnote{Note that we added a subscript $x$ to $\theta_{x,u}$ to emphasize that it is a coding morphism for return words to $u$ in $x$.}.
\end{lem}

\begin{thm}
\label{thm:morphic with return morphisms}
Let $x = \tau(\sigma^\omega(b))$ be an aperiodic uniformly recurrent morphic word.
One can algorithmically compute (from $(\sigma,\tau,b)$) two return morphisms $\theta:\cC^* \to \cC^*$ and $\lambda:\cC^* \to \cA^*$ such that $\theta$ is primitive and $x = \lambda(\theta^\omega(1))$.
\end{thm}
\begin{proof}
We define a sequence of derivated sequence of $x$ as follows.
Fix $x^{(0)} = x$ and $w_0 = x^{(0)}_0$, i.e., $w_0$ is the first letter of $x^{(0)}$.
For every integer $n \geq 1$, we set $x^{(n)} = \cD_{w_{n-1}}(x^{(n-1)})$ and $w_n = 1$. 
In other words, $x^{(1)}$ is the derived sequence of $x$ with respect to the first letter of $x$ and, for every $n \geq 2$, $x^{(n)}$ is the derived sequence of $x^{(n-1)}$ with respect to the letter 1 (which is the first letter of $x^{(n-1)}$).
By iterating Lemma~\ref{lemma:derived of derived}, we deduce that there is a sequence of words $u_n \in \cL(x)$ such that $\lim_{n \to +\infty} |u_n| = +\infty$ and $x^{(n)}$ is the derived sequence of $x$ with respect to $u_n$.

Every sequence $x^{(n)}$ is uniformly recurrent and, by Theorem~\ref{thm:durand_caclculalbe}, we can compute some morphisms $\theta_n, \sigma_n, \psi_n$ such that $x^{(n)} = \psi_n(\sigma_n^\omega(1)) = \theta_n(x^{(n+1)})$.
Still by Theorem~\ref{thm:durand_caclculalbe}, there exists $1 \leq m<n$ bounded by a computable constant such that $(\sigma_m,\psi_m) = (\sigma_n,\psi_n)$, so that $x^{(m)} = x^{(n)}$.
The construction being deterministic, this implies that for all $k \geq m$, we have $x^{(k+m-n)} = x^{(k)}$, so that $\theta_{k+m-n} = \theta_k$.
In particular, setting $\lambda = \theta_0 \circ \cdots \circ \theta_{m-1}$ and $\theta = \theta_m \circ \cdots \circ \theta_{n-1}$, we have $x = \lambda(x^{(m)})$ and $x^{(m)} = \theta(x^{(m)})$. 
In particular, $\theta(1)$ starts with 1.
Furthermore, both $\lambda$ and $\theta$ are return morphisms and, using Theorem~\ref{T:S-adic representation of minimal}, $\theta$ is primitive, hence prolongable on $1$. 
We finally get $x = \lambda(\theta^\omega(1))$.
\end{proof}

The next result directly follows from Theorem~\ref{T:main}, Theorem~\ref{T:ult dendric iff derived dendric} and Theorem~\ref{thm:caract ev dendric}.

\begin{thm}
\label{thm:morphic dendric}
Let $x = \tau(\sigma^\omega(b))$ be an aperiodic uniformly recurrent morphic word and let $\lambda,\theta$ be the computable morphisms given by Theorem~\ref{thm:morphic with return morphisms}.
\begin{enumerate}
\item
The shift space $X_x$ is eventually dendric if and only if $\theta$ is dendric and the sequence $(\theta,\theta,\dots)$ labels an infinite path in $\cG^L(\{\theta\})$ and in $\cG^R(\{\theta\})$.
\item
The shift space $X_x$ is dendric if and only if $\lambda$ and $\theta$ are dendric and the sequence $(\lambda,\theta,\theta,\dots)$ labels an infinite path in $\cG^L(\{\lambda,\theta\})$ and in $\cG^R(\{\lambda,\theta\})$.
\end{enumerate}
\end{thm}

\begin{prop}\label{P:decidability of the threshold}
Let $x = \tau(\sigma^\omega(b))$ be an aperiodic uniformly recurrent morphic word. If $X_x$ is eventually dendric, then we can compute the graphs $G^L(X_x)$ and $G^R(X_x)$ and the threshold for which it is eventually dendric.
\end{prop}
\begin{proof}
Let $\lambda,\theta$ be the computable morphisms given by Theorem~\ref{thm:morphic with return morphisms} and let $y = \theta^\omega(1)$. Let us first show that we can compute the graphs $G^L(X_y)$ and $G^R(X_y)$. If $X_x$ is eventually dendric, then $X_y$ is dendric by Theorem~\ref{thm:morphic dendric}. By Proposition~\ref{P:link eventually dendric with graphs}, there exists $K$ such that $G^L(X_y) = G^L_K(X_y)$ and $G^R(X_y) = G^R_K(X_y)$, i.e., the factors of $\cL_{\geq K-1}(X_y)$ are ordinary. Moreover, there exists $M$ such that each length-$(K-1)$ factor is an extended image of a $\theta$-initial factor by some $\theta^m$ with $m < M$. Let $Z$ be an Arnoux-Rauzy shift space over the same alphabet as $X_y$. Then every $w \in \cL_{K-1}(X_y)$ is a factor of $\theta^M \cdot Z$ with the same extension graph.
Observe conversely that every factor of $\theta^M \cdot Z$ is either a factor of $X_y$ (with the same extension graph) or an extended image of a factor of $Z$ by $\theta^M$. In the second case, as every factor of $Z$ is ordinary, so are their extended images under $\theta^M$ by Remark~\ref{R:extended images when ordinary graph}. This shows
that any factor of $\theta^M \cdot Z$ that is not in $X_y$ is ordinary. In particular, the elements of $\cL_{\geq K-1}(\theta^M \cdot Z)$ are ordinary. We then conclude that $G^L(X_y) = G^L_K(X_y) = G^L_K(\theta^M \cdot Z) = G^L(\theta^M \cdot Z)$, and similarly that $G^R(X_y) = G^R(\theta^M \cdot Z)$. Moreover, this remains true for any $M' \geq M$.

Let us study the sequence $(G^L(\theta^n \cdot Z))_{n \geq 0}$. Since $G^L(Z)$ is the complete graph, $G^L(\theta \cdot Z)$ is one of its subgraph.
By Proposition~\ref{P:image of eventually dendric with graphs}, $G^L(\theta^{n+1} \cdot Z) = \theta^L(G^L(\theta^n \cdot Z))$ for all $n$. This then implies by induction that $G^L(\theta^{n+1} \cdot Z)$ is a subgraph of $G^L(\theta^n \cdot Z)$ for all $n$. Moreover, if they are equal, then $G^L(\theta^{n+k} \cdot Z) = G^L(\theta^n \cdot Z)$ for all $k \geq 0$.

This shows that, to compute $G^L(X_y)$, one simply needs to iteratively compute the graphs $G^L(\theta^n \cdot Z)$ using Proposition~\ref{P:image of eventually dendric with graphs} (and starting from the complete graph) until they stabilize. 
We can similarly compute the graph $G^R(X_y)$. By Proposition~\ref{P:image of eventually dendric with graphs} for the morphism $\lambda$, we then obtain $G^L(X_x)$ and $G^R(X_x)$.

We now show that we can compute the dendricity threshold of $X_x$. As $X_y$ is dendric, if $G^L_N(X_y) = G^L(X_y)$ and $G^R_N(X_y) = G^R(X_y)$, then any word of $\cL_{\geq N}(X_y)$ is ordinary. Knowing $G^L(X_y)$ and $G^R(X_y)$, we can find such $N$ by iteratively computing the graphs $G^L_n(X_y)$ and $G^R_n(X_y)$ and comparing them to $G^L(X_y)$ and $G^R(X_y)$. By Proposition~\ref{P:image of eventually dendric}, this gives an upper bound on the dendricity threshold of $X_x$. We can then compute its precise value by checking all small factors.
\end{proof}

Note that, using the notations of the proof, $y$ is dendric if and only if the sequence $(\theta, \theta, \dots)$ labels the paths $(G^L(\theta^n \cdot Z))_{n \geq 0}$ in $\cG^L(\{\theta\})$ and $(G^R(\theta^n \cdot Z))_{n \geq 0}$ in $\cG^R(\{\theta\})$. In other words, the study of the sequences $(G^L(\theta^n \cdot Z))_{n \geq 0}$ and $(G^R(\theta^n \cdot Z))_{n \geq 0}$ also leads to an alternative method to decide (eventual) dendricity.

\begin{cor}
Given two morphisms $\sigma,\tau$ such that $X(\sigma,\tau)$ is minimal, it is decidable whether $X(\sigma,\tau)$ is eventually dendric of threshold $N$.
\end{cor}
\begin{proof}
By Lemma~\ref{lem:compute morphisms}, we can compute morphisms $\sigma',\tau'$ such that $X(\sigma,\tau) = X_x$, with $x = \tau'(\sigma'^\omega(b))$.
It is decidable whether $x$ is periodic and, if so, one can compute the minimal period~\citep{Durand:HD0L}.
In that case, $X(\sigma,\tau)$ is eventually dendric and the dendricity threshold is bounded by the minimal period.
If $x$ not periodic, then we use Theorem~\ref{thm:morphic dendric} and Proposition~\ref{P:decidability of the threshold} to conclude.
\end{proof}

\section{Interval exchanges}
\label{S:interval exchanges}

An interval exchange transformation is a pair $(\{I_a \mid a \in \cA\}, T)$, where $\cA$ is an alphabet, $\{I_a \mid a \in \cA\}$ is a partition of $[0,1)$ into left-closed and right-open intervals and $T:[0,1) \to [0,1)$ is a bijection whose restriction to each interval $I_a$ is a translation.
Another way of defining an interval exchange transformation is by means of a pair $\binom{\leq_1}{\leq_2}$ of total orders on $\cA$.
An interval exchange transformation is said to be {\em regular} if the orbits of the non-zero left endpoints of the intervals $I_a$ are infinite and disjoint.
We refer to~\citep{Ferenczi_Zamboni:2008} for more precise definitions.

\subsection{Extension graphs in codings of IET}

With an interval exchange transformation is associated a shift space, called the {\em natural coding} of $T$ and defined as the set of infinite words $(x_n)_{n \in \Z}$ over $\cA$ such that for every factor $u = u_0 \cdots u_{n-1}$ of $x$, the set $\bigcap_{0 \leq i < n} T^{-i}(I_{u_i})$ is non-empty.
The definition using total orders allows to give the following characterization of natural codings of regular interval exchange transformations.

\begin{thm}[Ferenczi and Zamboni~\citep{Ferenczi_Zamboni:2008}]
\label{T:Fer-Zam}
A shift space $X$ over $\cA$ is the natural coding of a regular interval exchange transformation with the pair of orders $\binom{\leq_1}{\leq_2}$ if and only if it is minimal, $\cA \subseteq \cL(X)$ and it satisfies the following conditions for every $w \in \cL(X)$:
\begin{enumerate}
\item\label{item:IE non crossing edges}
	for all $(a_1, b_1), (a_2, b_2) \in E_X(w)$, if $a_1 <_2 a_2$, then $b_1 \leq_1 b_2$;
\item\label{item:IE intersection}
	for all $a_1, a_2 \in E^L_X(w)$, if $a_1, a_2$ are consecutive for $\leq_2$, then $E^R_X(a_1w) \cap E^R_X(a_2w)$ is a singleton.
\end{enumerate}
In particular, for every $w \in \cL(X)$, $E_X^L(w)$ is an interval for $\leq_2$ and $E^R_X(w)$ is an interval for $\leq_1$.
\end{thm}

This result can be reformulated as follows.
Let $\leq^L$ and $\leq^R$ be two total orders on $\cA$. 
A factor $w$ of a shift space $X \subseteq A^\Z$ is said to be \emph{planar} for $(\leq^L, \leq^R)$ if item~\ref{item:IE non crossing edges} is satisfied for $\binom{\leq^R}{\leq^L}$.
This implies that, placing the left and right vertices of $\cE_X(w)$ on parallel lines and ordering them respectively by $\leq^L$ and $\leq^R$, the edges of $\cE_X(w)$ may be drawn as straight noncrossing segments, resulting in a planar graph. In particular, $\cE_X(w)$ is acyclic.

If $w$ is planar for $(\leq^L, \leq^R)$, item~\ref{item:IE intersection} is then equivalent to the fact that $\cE_X(w)$ is a tree. In particular, this proves that codings of regular interval exchanges are dendric shift spaces~\citep{bifix_IET}.

A shift space over $\cA$ is said to be \emph{planar} for $(\leq^L, \leq^R)$ if every $w \in \cL(X)$ is planar for $(\leq^L, \leq^R)$.
Theorem~\ref{T:Fer-Zam} then states that a shift space is the natural coding of a regular interval exchange transformation for the orders $\binom{\leq^R}{\leq^L}$ if and only if it is a minimal dendric shift space planar for $(\leq^L, \leq^R)$. Note that, in this case, it is known that the pair $(\leq^L, \leq^R)$ is \emph{irreducible}, i.e. for any $0 < k < \Card(\cA)$, the $k$ $\leq^L$-smallest letters are not the $k$ $\leq^R$-smallest letters.

The next result is classical when dealing with interval exchange transformations.
We give a combinatorial proof.
\begin{prop}
\label{P:IET ultimement 2 extensions}
Let $X$ be a minimal dendric shift over $\cA$ which is planar for $(\leq^L,\leq^R)$. 
\begin{enumerate}
\item
Any long enough left special factor $w$ is such that $E^L_X(w)$ is equal to $\{a,b\}$, where $a,b$ are consecutive for $\leq^L$. 
Furthermore, for every two $\leq^L$-consecutive letters $a,b$ and for all $n \in \N$, there is a (unique) left special factor $w$ of length $n$ such that $\{a,b\} \subseteq E_X^L(w)$. 
\item
The same holds on the right, i.e., any long enough right special factor $w$ is such that $E^R_X(w)$ is equal to $\{a,b\}$, where $a,b$ are consecutive for $\leq^R$. 
Furthermore, for every two $\leq^R$-consecutive letters $a,b$ and for all $n \in \N$, there is a (unique) right special factor $w$ of length $n$ such that $\{a,b\} \subseteq E_X^R(w)$. 
\end{enumerate} 
\end{prop}

\begin{proof}
We prove the first item.
Since $X$ is dendric, there exists $N$ satisfying the conditions of Proposition~\ref{P:equiv eventually dendric DP}. Let $w \in \cL_{\geq N}(X)$ be a left special factor.
By Theorem~\ref{T:Fer-Zam}, the letters of $E_X^L(w)$ are consecutive for $\leq^L$ and it suffices to prove that $\Card(E^L_X(w)) = 2$. Assume that there exist $a,b,c \in E^L_X(w)$ such that $a <^L b <^L c$.

We claim that $bw$ is not right special. Otherwise, by definition of $N$, $E^R_X(bw) = E^R_X(w)$ and, as there exist $a', c'$ such that $E^R_X(aw) = \{a'\}$ and $E^R_X(cw) = \{c'\}$, $wa'$ and $wc'$ are left special. However, this contradicts the definition of $N$ thus $a' = c'$. Using item~\ref{item:IE non crossing edges} of Theorem~\ref{T:Fer-Zam}, this implies that $E^R_X(bw) = \{a'\}$, a contradiction.

Thus $bw$ is not right special and there exists a unique letter $b'$ such that $bwb' \in \cL(X)$. By definition of $N$, $E^L(wb') = E^L_X(w)$ and we can iterate the reasoning to show that no word of $bw\cA^* \cap \cL(X)$ is right special. This is a contradiction as $X$ is minimal dendric.

We now prove the second part of item 1 by induction on $n$. It is true for $n = 0$ and if $u$ is the unique word of $\cL_n(X)$ such that $a, b \in E^L_X(u)$ then, by item~\ref{item:IE intersection} of Theorem~\ref{T:Fer-Zam}, there exists $c$ such that $E^R_X(au) \cap E^R_X(bu) = \{c\}$. Thus, $uc$ is the unique word of $\cL_{n+1}(X)$ having $a$ and $b$ as left extensions.
\end{proof}

\subsection{S-adic characterization of regular interval exchange transformations}

From Theorem~\ref{T:Fer-Zam}, we know that the natural coding of a regular interval exchange transformation is a minimal dendric shift.
By Theorem~\ref{T:main}, any of its $\cS$-adic representation made of return morphisms labels a path in $\cG^L(\cS)$ and in $\cG^R(\cS)$.
In this section, we characterize those labeled paths.

\begin{defi}
A \emph{line graph} on $\cA$ is a graph $G$ such that, if $\cA = \{a_1, \dots, a_n\}$, the edges are exactly the pairs $(a_i, a_{i+1})$, $i < n$. This graph is associated with the orders
\[
	a_1 < a_2 < \dots < a_n \quad \text{and} \quad a_n \prec a_{n-1} \prec \dots \prec a_1.
\]
We then define $G(\leq)$ and $G(\preceq)$ as the graph $G$.
\end{defi}

The following result is a direct consequence of Proposition~\ref{P:IET ultimement 2 extensions}. Its converse is false, as seen in Example~\ref{Ex:not dendric with acyclic graphs}.

\begin{cor}\label{C:interval exchanges have line graphs}
If $X$ is a minimal dendric shift planar for the orders $(\leq^L, \leq^R)$, then $G^L(X) = G(\leq^L)$ and $G^R(X) = G(\leq^R)$.
\end{cor}

If $\preceq$ and $\leq$ are two total orders on $\cA$, a partial map $\varphi : \cA \to \cA$ is \emph{order preserving} from $\preceq$ to $\leq$ if, for all $x, y \in \dom(\varphi)$, we have
\[
	x \prec y \Rightarrow \varphi(x) \leq \varphi(y).
\]

A return morphism $\sigma : \cA^* \to \cA^*$ is \emph{left order preserving} from $\preceq$ to $\leq$ if, for all $s \in \cT^L(\sigma)$, $\varphi^L_{\sigma, s}$ is order preserving from $\preceq$ to $\leq$.
Similarly, $\sigma$ is \emph{right order preserving} from $\preceq$ to $\leq$ if, for all $p \in \cT^R(\sigma)$, $\varphi^R_{\sigma, p}$ is order preserving from $\preceq$ to $\leq$.

\begin{lem}\label{L:unique order for order preserving}
For every return morphism $\sigma : \cA^* \to \cA^*$ and every total order $\leq$ on $\cA$, there exists a unique total order $\preceq$ on $\cA$ such that $\sigma$ is left (resp., right) order preserving from $\preceq$ to $\leq$.
\end{lem}

\begin{proof}
We prove the result for left order preserving. Let us begin with the existence of such an order $\preceq$. For all $s \in \Suff(\sigma(\cA))$, we will build an order $\preceq_{s}$ on $\cB_s = \{a \in \cA \mid \sigma(a) \in \cA^*s\}$\footnote{Note the difference with $\cA^L_{\sigma, s} = \{a \in \cA \mid \sigma(a) \in \cA^+s\}$.} such that, for all $s' \in \cA^*s$, the map $\varphi^L_{\sigma, s'}$ is order preserving from $\preceq_s$ to $\leq$. The conclusion will follow with $s = \eps$.

We proceed by induction on the length of $s$, starting with $|s|$ maximal. If $s$ is maximal, i.e. for all $a \in \cA$, $\cB_{as}$ is empty, then $s \in \sigma(\cA)$ and $\cB_s$ contains a unique element thus $\preceq_s$ is a trivial order.

Assume now that $s$ is not of maximal length and that we have the orders $\preceq_{as}$ for all $a \in \varphi^L_{\sigma, s}(\cA)$. Since $\sigma(\cA)$ is a suffix code, the sets $\cB_{as}$ form a partition of $\cB_s$ thus, we can define the order $\preceq_s$ on $\cB_s$ by $x \prec_s y$ if
\begin{enumerate}
\item\label{item:order preserving first case}
	$x, y \in \cB_{as}$ and $x \prec_{as} y$, or
\item\label{item:order preserving second case}
	$x \in \cB_{as}$, $y \in \cB_{bs}$ and $a < b$.
\end{enumerate}
For all $s' \in \cA^*as$, $\varphi^L_{\sigma, s'}$ is order preserving from $\preceq_{as}$ to $\leq$ thus it is order preserving from $\preceq_s$ to $\leq$.
In addition, if $x, y \in \cB_s$ are as in case~\ref{item:order preserving first case}, then
\[
	\varphi^L_{\sigma, s}(x) = a = \varphi^L_{\sigma, s}(y)
\]
and if they are as in case~\ref{item:order preserving second case}, then
\[
	\varphi^L_{\sigma, s}(x) = a < b = \varphi^L_{\sigma, s}(y).
\]
Thus, $\varphi^L_{\sigma, s}$ is also order preserving from $\preceq_s$ to $\leq$.

We now prove the uniqueness. Assume that $\sigma$ is left order preserving from $\preceq$ to $\leq$ and from $\preceq'$ to $\leq$ and let $x, y \in \cA$ be such that $x \prec y$ and $y \prec' x$. If $s = s_\sigma(x, y)$, then $\varphi^L_{\sigma, s}(x) \ne \varphi^L_{\sigma, s}(y)$. Since $\sigma$ is left order preserving from $\preceq$ to $\leq$, we have $\varphi^L_{\sigma, s}(x) < \varphi^L_{\sigma, s}(y)$. Since $\sigma$ is also left order preserving from $\preceq'$ to $\leq$, we have the converse inequality, which is a contradiction.
\end{proof}

\begin{ex}
Let us consider the morphism $\beta : 0 \mapsto 0, 1 \mapsto 01, 2 \mapsto 02, 3 \mapsto 032$ of Example~\ref{Ex:image of graph} and let us build the order $\preceq$ such that $\beta$ is left order preserving from $\preceq$ to $\leq$ where $3 < 0 < 2 < 1$. Using the notations of the previous proof, we have $3 \prec_2 2$ since $3 < 0$. We then have $0 \prec_\eps 3 \prec_\eps 2 \prec_\eps 1$ since $0 < 2 < 1$, and $\preceq$ is given by the order $\preceq_\eps$.
\end{ex}

We say that a return morphism $\sigma$ is \emph{planar preserving} from $(\preceq^L, \preceq^R)$ to $(\leq^L, \leq^R)$ if it is left order preserving from $\preceq^L$ to $\leq^L$ and right order preserving from $\preceq^R$ to $\leq^R$.
And it is $(\leq^L, \leq^R)$\emph{-planar} if all of the $\sigma$-initial factors are planar for $(\leq^L, \leq^R)$.
Note that this definition is completely independent on the choice of the word $w$ for which $\sigma$ codes the return words.

The terminology planar preserving comes from the following results.

\begin{prop}
\label{P:planar morphisms}
Let $X$ be a shift space and $\sigma$ be a return morphism planar preserving from $(\preceq^L, \preceq^R)$ to $(\leq^L, \leq^R)$. A word $v \in \cL(X)$ is planar for $(\preceq^L, \preceq^R)$ if and only if every extended image of $v$ under $\sigma$ is planar for $(\leq^L, \leq^R)$.
\end{prop}

\begin{proof}
We prove that the graph $\cE_X(v)$ has two crossing edges for $(\preceq^L, \preceq^R)$ if and only if there exists an extended image $u$ such that $\cE_{\sigma \cdot X}(u)$ has two crossing edges for $(\leq^L, \leq^R)$.
Let $(x_1, y_1)$ and $(x_2, y_2)$ be two bi-extensions of $v$ such that $x_1 \ne x_2$ and $y_1 \ne y_2$. If $s = s_\sigma(x_1, x_2)$ and $p = p_\sigma(y_1, y_2)$, we denote
\[
	x'_1 = \varphi_{\sigma, s}^L(x_1), \quad
	x'_2 = \varphi_{\sigma, s}^L(x_2), \quad
	y'_1 = \varphi_{\sigma, p}^R(y_1) 
	\quad \text{and} \quad 
	y'_2 = \varphi_{\sigma, p}^R(y_2).
\]
By construction, $(x'_1, y'_1)$ and $(x'_2, y'_2)$ are two bi-extensions of the extended image $u := s\sigma(v)p$ and are such that $x'_1 \ne x'_2$ and $y'_1 \ne y'_2$.
Remark that, for any such pair of bi-extensions of an extended image $u'$ of $v$, it is possible to find a corresponding pair of bi-extensions of $v$.

To conclude, it suffices to prove that $(x_1, y_1)$ and $(x_2, y_2)$ are crossing edges in $\cE_X(v)$ if and only if $(x'_1, y'_1)$ and $(x'_2, y'_2)$ are crossing edges in $\cE_{\sigma \cdot X}(u)$.
Assume that $x_1 \prec^L x_2$. Since $\sigma$ is planar preserving, $\varphi^L_{\sigma, s}$ is order preserving from $\preceq^L$ to $\leq^L$ and we have $x'_1 <^L x'_2$. Moreover, $\varphi^R_{\sigma, p}$ is order preserving from $\preceq^R$ to $\leq^R$ thus
\[
    y_1 \prec^R y_2 \Leftrightarrow y'_1 <^R y'_2.
\]
This ends the proof.
\end{proof}

\begin{cor}\label{C:planar preserving}
Let $X$ and $Y$ be two dendric shifts such that $Y = \sigma \cdot X$ where $\sigma$ is $(\leq^L, \leq^R)$-planar and planar preserving from $(\preceq^L, \preceq^R)$ to $(\leq^L, \leq^R)$. The shift $X$ is planar for $(\preceq^L, \preceq^R)$ if and only if $Y$ is planar for $(\leq^L, \leq^R)$.
\end{cor}

We can now prove a result similar to Theorem~\ref{T:main} but in the case of interval exchanges.
This time, the characterization only uses one graph. Observe that, if $\leq^*$ denotes the reverse order of $\leq$, then the planar shift spaces (or codings of interval exchange transformations) for $(\leq^L, \leq^R)$ and for $((\leq^L)^*,(\leq^R)^*)$ coincide. Therefore, to avoid duplicates in the graph, we fix $a <^L b$ for some $a,b \in \cA$.

\begin{defi}\label{D:graph interval exchanges}
Let $a,b$ be two distinct letters in $\cA$.
Let $\cS$ be a set of dendric return morphisms from $\cA^*$ to $\cA^*$. The graph $\cG_{IET}(\cS)$ is such that its vertices are the irreducible pairs of orders $(\leq^L,\leq^R)$ such that $a <^L b$, and there is an edge from $(\leq^L,\leq^R)$ to $(\preceq^L,\preceq^R)$ labeled by $\sigma \in \cS$ if 
\begin{enumerate}
\item
	$(G(\leq^L), \sigma, G(\preceq^L))$ is an edge of $\cG^L(\cS)$;
\item
	$(G(\leq^R), \sigma, G(\preceq^R))$ is an edge of $\cG^R(\cS)$;
\item
	$\sigma$ is $(\leq^L, \leq^R)$-planar;
\item
	$\sigma$ is planar preserving from $(\preceq^L, \preceq^R)$ to $(\leq^L, \leq^R)$, or from $((\preceq^L)^*, (\preceq^R)^*)$ to $(\leq^L, \leq^R)$.
\end{enumerate}
\end{defi}

By Lemma~\ref{L:unique order for order preserving}, the graph $\cG_{IET}$ is deterministic in the sense that, for each vertex and each morphism, there is at most one edge labeled by the morphism leaving the vertex.

\begin{thm}
\label{T:iet in G}
    Let $\cS$ be a family of dendric return morphisms from $\cA^*$ to $\cA^*$ and let $X$ be a shift space having an $\cS$-adic representation $\bsigma = (\sigma_n)_{n \geq 0}$. Then $X$ is the coding of a regular interval exchange if and only if $\bsigma$ is primitive and labels an infinite path in the graph $\cG_{IET}(\cS)$. Moreover, if this path starts in the pair $(\leq^L, \leq^R)$, then $X$ is the coding of a regular interval exchange for the orders $\binom{\leq^R}{\leq^L}$.
\end{thm}

\begin{proof}
Let us assume that $X$ is the coding of a regular interval exchange for the orders $\binom{\leq^R_0}{\leq^L_0}$, we can moreover assume that $a <^L_0 b$, where $a$ and $b$ are the letters used to define $\cG_{IET}(\cS)$. Then, $\sigma_0$ is $(\leq^L_0, \leq^R_0)$-planar, and by Lemma~\ref{L:unique order for order preserving}, there exist two orders $\leq^L_1$ and $\leq^R_1$ such that $a <^L_1 b$ and $\sigma$ is planar preserving from $(\leq^L_1, \leq^R_1)$ to $(\leq^L_0, \leq^R_0)$ or from $((\leq^L_1)^*, (\leq^R_1)^*)$ to $(\leq^L_0, \leq^R_0)$.

As $X_\bsigma^{(1)}$ is dendric, by Corollary~\ref{C:planar preserving}, it is the coding of a regular interval exchange for the orders $\binom{\leq^R_1}{\leq^L_1}$. In particular, the pair $(\leq^L_1,\leq^R_1)$ is irreducible. Moreover, $\sigma_0$ labels an edge from $G(\leq^L_0)$ to $G(\leq^L_1)$ in $\cG^L(\cS)$ and from $G(\leq^R_0)$ to $G(\leq^R_1)$ in $\cG^R(\cS)$. Therefore, $\sigma_0$ labels an edge from $(\leq^L_0, \leq^R_0)$ to $(\leq^L_1,\leq^R_1)$ in $\cG_{IET}(\cS)$. We then iterate this reasoning to show that $\bsigma$ labels an infinite path in $\cG_{IET}(\cS)$.

Assume now that $\bsigma$ is primitive and labels a path $((\leq^L_n, \leq^R_n))_{n \geq 0}$ in $\cG_{IET}(\cS)$. Since $\bsigma$ also labels an infinite path in $\cG^L(\cS)$ and $\cG^R(\cS)$, $X$ is minimal dendric. Let us prove that it is planar for $(\leq^L_0, \leq^R_0)$.
The $\sigma_0$-initial factors of $X$ are planar for $(\leq^L_0, \leq^R_0)$ since $\sigma_0$ is $(\leq^L_0, \leq^R_0)$-planar.
Let $u \in \cL(X)$ be a non-$\sigma_0$-initial factor of $X$. There exist a unique $k > 0$ and a unique $v \in \cL(X_\bsigma^{(k)})$ $\sigma_k$-initial such that $u$ is an extended image of $v$ by $\sigma_0 \dots \sigma_{k-1}$ by Proposition~\ref{P:definition of antecedent}. However, $\sigma_k$ is $(\leq^L_k, \leq^R_k)$-planar thus $v$ is planar for these orders. Iterating Proposition~\ref{P:planar morphisms} $k$ times allows us to conclude that $u$ is planar for $(\leq^L_0, \leq^R_0)$. As this is true for any $u \in \cL(X)$, $X$ is the coding of a regular interval exchange for the orders $\binom{\leq^R_0}{\leq^L_0}$.
\end{proof}

As in the dendric case, we can reduce the number of vertices of this graph using permutation. As a first step, we can completely fix the right (or top) order. However we can reduce this graph even more. For example, on the alphabet $\{0,1,2,3\}$, the irreducible pairs $\binom{0<1<2<3}{1<3<2<0}$ and $\binom{0<1<2<3}{3<1<0<2}$ are equivalent up to permutation even if they have the same top order. Indeed, $\binom{0<1<2<3}{3<1<0<2}$ is equivalent to $\binom{3<2<1<0}{2<0<1<3}$ which becomes $\binom{0<1<2<3}{1<3<2<0}$ with a simple permutation of the letters. In particular, this then gives a graph with $2$ vertices on a ternary alphabet, as done in~\citep{Gheeraert_Lejeune_Leroy:2021}, and with $9$ vertices for an alphabet of size 4.

\section{Future work}

Given a set $\cS$ of dendric return morphisms, Theorem~\ref{T:main} characterizes the $\cS$-adic shifts that are minimal and dendric.
We thus directly obtain the following corollary.

\begin{cor}
Let $\cS$ be a set of dendric return morphisms defined on some alphabet $\cA$ such that any minimal dendric shift over $\cA$ admits an $\cS$-adic representation.
Then a shift space $X \subseteq \cA^\mathbb{Z}$ is minimal and dendric if and only if it has a primitive $\cS$-adic representation labeling paths in the graphs $\cG^L(\cS)$ and $\cG^R(\cS)$.  
\end{cor}

This corollary is however not completely satisfactory.
Indeed it is in general difficult to find such a set $\cS$ of dendric return morphisms.
Recall that Theorem~\ref{T:decoding dendric} ensures that all morphisms in $\cS$ can be assume to be tame, i.e., are compositions of permutations and of the elementary morphisms
\[
	R_{ab}:
        \begin{cases}
		a \mapsto ab, \\
		c \mapsto c, & \text{for } c \neq a,
	\end{cases}
\qquad \text{and} \qquad
	L_{ab}:
        \begin{cases}
		a \mapsto ba, \\
		c \mapsto c, & \text{for } c \neq a.
	\end{cases}
\]

Over the alphabet $\cA = \{0,1,2\}$, a sufficient~\citep{Gheeraert_Lejeune_Leroy:2021} set of morphisms is given by the permutations and by 
\[
    \begin{array}{lll}
    \alpha : 
    \begin{cases}
        0 \mapsto 0 \\
        1 \mapsto 01 \\
        2 \mapsto 02
    \end{cases}
    & \beta : 
    \begin{cases}
        0 \mapsto 0 \\
        1 \mapsto 01 \\
        2 \mapsto 021
    \end{cases}
    & \gamma : 
    \begin{cases}
        0 \mapsto 0 \\
        1 \mapsto 01 \\
        2 \mapsto 012
    \end{cases}
    \\
    \delta_n : 
    \begin{cases}
        0 \mapsto 0 \\
        1 \mapsto 012^n \\
        2 \mapsto 012^{n+1}
    \end{cases}
    & 
    \zeta_n : 
    \begin{cases}
        0 \mapsto 02^n \\
        1 \mapsto 01 \\
        2 \mapsto 02^{n+1}
    \end{cases}
    & \eta : 
    \begin{cases}
        0 \mapsto 02 \\
        1 \mapsto 01 \\
        2 \mapsto 012
    \end{cases}
    \end{array}
\]

Furthermore, there exists a regular language over the alphabet of elementary morphisms that exactly corresponds to the set $\cS_3 = \{\alpha,\beta,\gamma,\delta_n, \zeta_n, \eta \mid n \in \mathbb{N}\}$:
\begin{align*}
\alpha = L_{10}L_{20}
\quad 
& \beta = L_{20}R_{21}L{10}
\quad 
\gamma = L_{10}L_{21}
\\ 
\delta_n = L_{10}R_{12}^nL{21}
\quad 
& 
\zeta_n = L_{10}R_{02}^nL{20}
\quad 
\eta = L_{10}R_{02}L{21}
\end{align*}

For more general alphabets, it is also possible to find regular languages over the alphabet of elementary morphisms, but only for restricted classes of dendric shifts.
Such a regular language is for instance trivially described for Arnoux-Rauzy shifts, and can be obtain from the Rauzy diagrams for regular interval exchange transformations.
The Rauzy diagrams encode the behavior of interval exchanges under the Rauzy induction, and this induction process transfers to the codings by means of elementary morphisms~\citep{Rauzy:79}.
For codings of regular IET, derivation with respect to a word $w$ corresponds to induction on a particular interval associated with $w$, and this induction can be performed by a sequence of Rauzy inductions~\citep{Dolce_Perrin:interval}.
Hence the return morphisms appearing in codings of $k$-regular interval exchange form a regular language over the alphabet of elementary morphisms and an automaton can be deduced from the Rauzy diagram.

This raises the question of finding a regular language for the general case. 
More precisely, if $\cA$ is fixed, is there a regular language over the alphabet of elementary morphisms over $\cA$ that exactly corresponds to dendric return morphisms over $\cA$?

\bibliography{biblio.bib}

\end{document}